\documentclass[11pt]{amsart}
\usepackage{amssymb, amsmath, amsmath}
\usepackage[backref]{hyperref}
\usepackage{mathrsfs}
\usepackage{amscd}   
\usepackage{float} 
\usepackage{fullpage}
\usepackage[all]{xy} 
\usepackage{setspace}
\usepackage[margin=1.86  cm, headsep = 0.1 cm, footskip = 0.7 cm]{geometry}
\usepackage[english]{babel}
\usepackage{amsfonts}
\usepackage{xurl}
\usepackage[T1]{fontenc}
\usepackage{comment}
\usepackage{enumitem}
\usepackage{tikz-cd}

\newcommand\ZZ{\mathbb{Z}}
\newcommand{\Z}{\mathbb{Z}}
\newcommand\NN{\mathbb{N}}

\newcommand\RR{\mathbb{R}}
\newcommand\R{\mathbb{R}}
\newcommand\QQ{\mathbb{Q}}
\newcommand\Qbar{{\overline{\QQ}}}

\newcommand\FF{\mathbb{F}}

\newcommand\PP{\mathbb{P}} 
\newcommand{\A}{\mathbb{A}} 
\newcommand\GG{\mathbb{G}} 
\newcommand{\OO}{\mathcal{O}} 
\newcommand{\fdiscn}{f^{\Delta,n}} 
\newcommand{\ol}{\overline} 
\newcommand{\countingESZB}{\mathcal{N}_{\mathcal{U}, \mathcal{V}} (B) } 
\newcommand{\JetsX}{\mathcal{J}_0\mathcal{X}} 
\newcommand{\Jets}{\mathcal{J}_0} 
\newcommand{\locdiscrep}{\delta_{\mathcal{V}; v}(x)}

\DeclareMathOperator{\spec}{Spec} 
\DeclareMathOperator{\aut}{Aut} 
\DeclareMathOperator{\Cl}{Cl} 
\DeclareMathOperator{\gal}{Gal} 
\DeclareMathOperator{\GL}{GL} 
\DeclareMathOperator{\Hom}{Hom} 
\DeclareMathOperator{\height}{ht} 
\DeclareMathOperator{\Height}{Ht} 
\DeclareMathOperator{\Stab}{Stab} 
\DeclareMathOperator{\coker}{\mathrm{coker}}
\DeclareMathOperator{\edd}{edd} 
\DeclareMathOperator{\Nm}{Nm} 
\DeclareMathOperator{\im}{Im} 

\newcommand{\Gal}[1]{\Gamma_{#1}} 

\newtheorem{theorem}{Theorem}
\newtheorem{proposition}[theorem]{Proposition}
\newtheorem{lemma}[theorem]{Lemma}

\theoremstyle{definition}
\newtheorem{definition}[theorem]{Definition}

\newtheorem{notation}[theorem]{Notation}

\newtheorem{conjecture}[theorem]{Conjecture}

\newtheorem{example}[theorem]{Example}

\theoremstyle{remark}
\newtheorem{remark}[theorem]{Remark}

\numberwithin{theorem}{section}
\numberwithin{equation}{section}

\title{On rational points on classifying stacks and Malle's conjecture}

\author{Shabnam Akhtari}
\address{Pennsylvania State University, Department of Mathematics,  University Park, PA 16802 USA}
\email{akhtari@psu.edu}
\author{Jennifer Park} 
\address{The Ohio State University, Department of Mathematics, Columbus, OH 43210, USA}
\email{park.2720@osu.edu}
\author{Marta Pieropan} 
\address{Utrecht University, Mathematical Institute, Budapestlaan 6, 3584 CD Utrecht, The Netherlands}
\email{m.pieropan@uu.nl}
\author{Soumya Sankar}
\address{Utrecht University, Mathematical Institute, Budapestlaan 6, 3584 CD Utrecht, The Netherlands}
\email{s.sankar@uu.nl}
\date{\today}

\begin{document}

\maketitle

\begin{abstract}
    In this expository article, we compare Malle's conjecture on counting number fields of bounded discriminant with recent conjectures of Ellenberg--Satriano--Zureick-Brown and Darda--Yasuda on counting points of bounded height on classifying stacks. We illustrate the comparisons via the classifying stacks \(B(\ZZ/n\ZZ)\) and \(B\mu_n\).  
\end{abstract}

\section{Introduction}
\sloppy
In 2004, Malle \cite{malle1,malle2} put forward a conjecture for the asymptotic growth of the count of number field extensions of bounded discriminant and fixed Galois group. Field extensions of fixed Galois group correspond to certain points on suitable classifying stacks. Recently, Ellenberg, Satriano and Zureick-Brown \cite{ESZB} proposed a framework for heights on stacks and a conjecture for the order of growth of the number of rational points of bounded height on stacks that is compatible with Manin's conjecture \cite{FMT89, Batyrev-Manin} for varieties. They interpret Malle's conjecture as a special case of their conjecture in the case of classifying stacks of finite constant group schemes. Later, Darda and Yasuda \cite{darda-yasuda-torsors1}, \cite{darda-yasuda22})
proposed a different framework for heights and an asymptotic formula that extends Malle's conjecture to classifying stacks of not necessarily constant finite group schemes. 

The goal of this paper is to clarify the relation between these three conjectures over number fields focusing on the special cases of $B(\ZZ/n\ZZ)$ and $B\mu_n$. While several of the results and conjectures mentioned may be generalized to global function fields, we will not consider that case in this article. A working knowledge of algebraic number theory and familiarity with the basic language of schemes is required, but the reader is not expected to have any background on stacks.

The structure of the paper is as follows. In \S \ref{sec:BG}, we start by recalling the theory of classifying stacks of finite group schemes. We then describe in detail the sets of rational points on  $B\mu_n$ in \S \ref{subsec:mu_n} and on $B(\ZZ/n\ZZ)$ in \S \ref{subsec:Z/nZ}. In \S \ref{sec:heights}, we summarize the literature of heights on classifying stacks of finite groups, with particular focus on the \cite{ESZB} framework. Finally, in \S \ref{sec:malles-conjecture}, we overview
the literature on Malle's conjecture and highlight the connections to the two conjectures for stacks and to various results in the same spirit.
Along the whole paper, the 
concepts that are introduced, though often more general, are illustrated using the examples of $B(\ZZ/n\ZZ)$ and $B\mu_n$.

\subsection{Notation}
\label{sec:notation}

Throughout the paper, we will use \(K\) to denote a number field and \(\OO_K\) to denote the ring of integers of \(K\).  We will use \(M_K\) to denote the set of places of \(K\). For a finite \(v \in M_K\), \(q_v\) will denote the size of the residue field at \(v\). If \(v \in M_K\) is infinite,  \(q_v\) will denote the Euler number $e$. We will use \(S\) to denote a scheme, and in most cases, we will only be interested in \(S\) of the form \(\spec \ZZ\), \(\spec \OO_K\), or \(\spec K\).

A  {group scheme} over \(S\) is a group object in the category of schemes. Throughout this article, we will assume that our group schemes are finite and smooth. 
In particular, they are of the form $\spec A$ for some ring \(A\).
If a group scheme \(G_S\) (or just \(G\) when the base scheme is clear from context) corresponds to a constant sheaf of abelian groups, it is called a constant group scheme. A finite constant group scheme thus comes from some finite group \(G\) and we will denote it by \(\underline{G}\). 

Throughout this article, we will fix a positive integer \(n\). Let \(\ZZ/n\ZZ\) denote the additive cyclic group of order \(n\), and \(\mu_n\) the multiplicative group scheme of order $n$. Their group scheme structures are explained in \S\ref{sec:BG}. For every positive integer \(m\), we will use \(\zeta_m\) to denote a primitive \(m\)-th root of unity.

For a number field \(K\) we let \(\Gal{K} = \gal(\ol{K}/K)\) be the Galois group of a fixed algebraic closure $\ol{K}$ of $K$. For a finite group \(G\), an \'etale \(G\)-algebra \(L/K\) is an \'etale algebra  of degree \(\#G\) where \(G \hookrightarrow \mathrm{Aut}_K(L)\) and \(G\) acts transitively on the idempotents of \(L\) (this is the notion of a \(G\) structured algebra as in \cite{wood-probabilities}). When \(L\) is a field, it is Galois with Galois group \(G\).
For an extension of \'etale algebras \(L/K/\QQ\), \(\Delta(L/K)\) will denote the relative discriminant, while \(|\Delta(L/K)|\) will denote the absolute discriminant, i.e., the absolute value of the norm of \(\Delta(L/K)\) from \(K\) to \(\QQ\) of \(\Delta(L/K)\).

If \(L = K(\alpha)\) is an extension of degree \(n\), then its Galois group \(G\) acts transitively on the \(n\) conjugates of \(\alpha\) and hence gives a map \(G \hookrightarrow S_n\). In particular, if \(L\) is Galois over $K$, we get a map \(G \to S_{\#G}\). The regular representation of \(G\) is the map \(G \to S_{\#G}\) induced by the action of \(G\) acting on itself by left multiplication.

Given a group scheme \(G\) over \(S\), we will write \(BG_S\) for the classifying stack of \(G\) over \(S\), defined in \S \ref{subsec:BG}. If \(S\) is clear from context, we will suppress the subscript. Our assumptions on \(G\) imply that \(BG_S\) is an algebraic stack, and in the case when \(G\) is \'etale over \(S\), also Deligne--Mumford. See \cite[E.g. 8.1.12, Defn 8.3.1]{Olsson-Book}.
We will not introduce any technical definitions of stacks unless absolutely necessary, but we refer to the reader to \cite{fantechi-stacks}, \cite{voight-stacks}, \cite{Olsson-Book} and \cite{alpernotes} for further reading.

Given a scheme \(X\) or a stack \(\mathcal{X}\) over a field \(K\) and a height function \(\Height\) on \(\mathcal{X}(\ol{K})\), we say the function \(\Height\) has the Northcott property if for any real \(B>0\) and \'etale algebra \(L/K\), 
\[
\#\{x \in \mathcal{X}(L) \mid \Height(x) < B\} < \infty.
\]
In general, if \(\Height\) denotes some height function, we will use lowercase \(\height\) to denote the logarithmic height \(\log \Height\), and vice versa.

Given two functions \(f, g : \,\RR \to \RR\), we say that \(f \ll g\) if there is a constant \(c\) such that for \(x\) sufficiently large, \(f(x) \le cg(x)\). We say that \(f \sim g \) if \[
\lim_{x \to \infty } \frac{f(x)}{g(x)} = 1.
\]

\subsection*{Acknowledgements}
The authors would like to thank Brandon Alberts and Remy van Dobben de Bruyn for useful conversations about parts of the paper. Thanks go also to Aaron Landesman for useful comments, and to the anonymous referee for suggestions that improved the exposition and for coming up with the example that is now in Remark 2.18.
The writing of this article was initiated at the Women in Numbers 6 workshop at the Banff International Research Station in Banff, Alberta, Canada. The authors are grateful for having the opportunity to participate in this workshop.
SA was partially supported by the NSF award DMS-2327098. JP was partially supported by the NSF award DMS-2152182.
MP was partially supported by the Dutch Research Council (NWO) grants
VI.Vidi.213.019 and OCENW.XL21.XL21.011. SS was partially supported by the Dutch Research Council (NWO) grant OCENW.XL21.XL21.011. 
For the purpose of open access, a CC BY public copyright license is applied to any Author Accepted Manuscript version arising from this submission.

\section{Rational points on $BG$}
\label{sec:BG}

In this section, $G$ is a finite
group scheme over $\spec \ZZ$. 
In particular, it is affine. 
In many cases in this article, $G$ will be induced from a finite group (also denoted $G$ via abuse of notation) as the constant sheaf $\underline G$, which is representable over $\spec\ZZ$ by the scheme 
$$
G_{\ZZ}=\bigsqcup_{g\in G}\spec\ZZ,
$$
see \cite[\href{https://stacks.math.columbia.edu/tag/047F}{Tag 047F}]{stacks-project} and \cite[Above Corollary II.1.7]{MR559531}. Notice that $G_\ZZ$ is a finite disjoint union of $\#G$ $\ZZ$-points in $\A^1_\ZZ$. It is a smooth affine scheme. For every ring $A$ we write $G_A:=G_\ZZ\times_{\spec\ZZ}\spec A$.

We will define the classifying stack $BG$, which parametrizes \(G\)-torsors. We begin by recalling the definition of $G$-torsors in \S \ref{subsec:torsors}, the categorical definition of $BG$ in \S \ref{subsec:BG}, and we give explicit descriptions of points on this stack in the cases where \(G\) is \(\mu_n\) and \(\ZZ/n\ZZ\) in \S\S \ref{subsec:mu_n}-\ref{subsec:Z/nZ}.

In this section, we will state everything for the base scheme \(\spec \ZZ\). Several definitions work for a more general base scheme, but this makes our exposition easier.

\subsection{\(G\)-Torsors}
\label{subsec:torsors}

\begin{definition}
    Let $T$ be a scheme (over $\spec \ZZ$). A  {$G$-torsor over $T$} (which, in this case of affine group schemes is the same as a principal $G$-bundle) is a map of schemes $\pi \colon P \to T$ with a left $G$-action on $P$ such that on some (\'etale or fppf) cover $\{ U_i \to T\}$, $P \times_T U_i$ is the trivial torsor $G|_{U_i} \times U_i$. Equivalently, we can require that the map $G \times_T P \to P \times_T P$ given by $(g,p) \mapsto (p, gp)$ is an isomorphism by \cite[Proposition 4.5.6]{Olsson-Book}.
\end{definition}

{We refer the reader to \cite[\S 4.5]{Olsson-Book}, \cite[\S III.4]{MR559531}, and \cite{skorob-torsors} for further reading on torsors and their properties.}

\begin{remark}
If $T = \spec K$ for a field $K$, and $G = \spec A$ is finite, then a $G$-torsor over $T$ is a morphism $P \to \spec K$ with an action of $G$ on $P$ and such that under suitable field extensions $L_i/K$, $P_{L_i}\cong G_{L_i}$ compatible with the $G_{L_i}$-action by 
{left} multiplication on $G_{L_i}$.

\end{remark}

We now look at two special cases.

 \begin{example}[$G$ constant]
 \label{example:BG constant}
    
    In the case where the group scheme is a constant scheme given by a finite group, the set of $G$-torsors over $T = \spec K$ is in bijection with the set of \'etale \(G\)-algebras over $K$ (in the sense of \S \ref{sec:notation}). Furthermore, the set of  {connected} $G$-torsors over $T = \spec K$ is in bijection with the set of  {Galois} extensions of $K$ with Galois group $G$. 
    
    Indeed, let \(P = \spec L\) be a \(G\)-torsor over \(K\). Since the map $G\times P\to P\times P$ is an isomorphism, we have $L\otimes_KL\cong L^{\#G}$ as $L$-vector spaces, and hence, $[L:K]=\#G$. Since the map $G\times P\to P\times P$ is an isomorphism, $G$ acts on $P$ by automorphisms of $P$ as a $K$-scheme, or equivalently of $L$ as a $K$-algebra. Hence, there is a group homomorphism $G\to\aut_K(L)$. Let $H$ be its image. Then trivializing $P$ over $L$ and taking pullback induce a sequence of group homomorphisms
    $$
    G\to H\to \aut_{G\text{-torsors}/L}(P_L)\cong \aut_{G\text{-torsors}/L}(G_L)=G
    $$
    whose composition is the identity on $G$. Hence, $[L:K]=\#G\leq\#\aut_K(L)$, and $L/K$ is a Galois extension. 
    \end{example}

\begin{example} [$G$ \'etale-locally constant] 
    Assume that over a separable closure $\overline K$ of $K$, $G_{\overline K}$ is a constant group scheme. Write $P=\spec L$. Then
    there exist suitable field extensions $L'/K$ such that $L' \otimes_K L \cong L'^{\oplus \#G}$. 
    
        For example, if $G = \mu_3$,
    $K=\QQ$ and $L = \QQ(\sqrt[3]{2})$. Let $L' = L(\zeta_3)$. Then $P_{L'} = \spec(L \otimes_K L')$, and
    \begin{align*}
         L \otimes_K L' = \QQ(\sqrt[3]{2}) \otimes_\QQ \QQ(\sqrt[3]{2}, \zeta_3) &\cong \QQ[x]/(x^3-2) \otimes_\QQ \QQ(\sqrt[3]{2}, \zeta_3)\\
         &\cong \QQ(\sqrt[3]{2}, \zeta_3)[x]/(x^3-2) \cong \QQ(\sqrt[3]{2}, \zeta_3)^{\oplus 3},
    \end{align*}
   where the last isomorphism holds because $L'$ contains all the cube roots of 2.
   We will see in \S\ref{subsec:comparison} that $\mu_3$ becomes constant after base change to \(\QQ(\zeta_3)\).
   \end{example}

\subsection{The classifying stack $BG$}
\label{subsec:BG}

Let $G$ be a group scheme that is finite
over a base scheme $S$ and
acts trivially on $S$. The classifying stack $BG$ over $S$ is the quotient stack $[S/G]$. In this section, we briefly recall the definition of a quotient stack, and give interpretations of rational points of $BG$ in some special cases. As stated before, we work with \(S = \spec \ZZ\) for ease of exposition. Since \(BG_S\) is the base change of \(BG_{\spec \ZZ}\) to \(S\) for a general scheme \(S\), this reduction is harmless for this article. We refer the reader to \cite[\S 8.1]{Olsson-Book} for a more detailed exposition.

\subsubsection{Quotient stacks}
Let $X$ be a noetherian scheme. Let $G$ be an affine 
group scheme with a left 
action on $X$ given by $\rho: G \times X \to X$. 

\begin{definition}
    \label{defn:quot-stack}
    The quotient stack $[X/G]$ is the (pseudo)functor from the category of schemes to the category of groupoids,
\[
[X/G]: (Sch) \to (Gpds)
\]
such that 
\begin{enumerate}
\item \label{defn:quot-stack-1} On the level of objects, a scheme $T$ is sent to the category $[X/G](T)$, whose objects are the $G$-torsors $\pi: \mathcal E \to T$ together with a $G$-equivariant morphisms $\alpha: \mathcal E \to X$; and the morphisms of this category are isomorphisms of $G$-torsors commuting with the $G$-equivariant morphisms.
\item \label{defn:quot-stack-2} The morphisms are naturally induced by pullbacks along the morphism $f: T' \to T$ in $(Sch)$.
\end{enumerate}
\end{definition}
By setting $X = \spec \ZZ$ with the trivial $G$-action, we obtain the classifying stack $BG_{\spec \ZZ}:=[X/G]$ \cite[Definition 8.1.14]{Olsson-Book}.

\begin{notation}
\label{notation:BG-points}
    Although \(BG\) is a stack valued in groupoids, by abuse of notation, we will use \(BG(T)\) to denote the (pointed) set of isomorphism classes of equivariant \(G\)-torsors over \(T\). 
\end{notation}

When working with \(BG_{\spec \ZZ}\), the condition of $\alpha$ being $G$-equivariant becomes automatic. 
Indeed, \(\alpha\) must be the unique map \(\mathcal{E} \to \spec \ZZ\), which must in turn agree with the composition of the two \(G\)-equivariant maps \(\pi : \mathcal{E} \to T\) and the structure map \(T \to \spec \ZZ\).
In particular, this implies that \(BG_{\spec \ZZ}(T)\) is the set of isomorphism classes of \(G\)-torsors over \(T\). Further, it is well known 
that the latter set is classified by \(H^1_{fppf}(T, G)\). See for instance, \cite[\S 2.2.]{skorob-torsors} or \cite[Chapter 12]{Olsson-Book}.
In what follows, we will omit the subscript \(\spec \ZZ\) from \(BG_{\spec \ZZ}\).

\begin{example}[Twisting by Galois descent, see {\cite[\S 2.1]{skorob-torsors}} or {\cite[\S III.1]{serre-gal-coh}}]
\label{example:twisting-by-Galois}
    Let \(T  = \spec K \) and \(G\) a smooth finite 
    group scheme
    over $K$. 
    By the observation above and using the fact that $G$ is smooth, 
    we may classify the \(G\)-torsors over \(T\) by \(H^1_{fppf}(\spec K, G) = H^1_{\textit{\'et}}(\spec K, G)\), which in turn is equal to the Galois cohomology group \(H^1(\Gal{K}, G(\overline{K}))\). Given a \(G\)-torsor \(P\), and a continuous cocycle \(\varphi: \Gal{K} \to G(\overline{K})\), a standard way to obtain the twist of \(P\) by \(\varphi\) is by taking invariants of the twisted  Galois action,
    \[
    \Gal{K} \times P_{\overline{K}} \to P_{\overline{K}}, \qquad (\sigma, x) \mapsto \varphi(\sigma) \cdot (\sigma(x)).
    \]
    Conversely, any other \(G\) torsor over \(K\) can be obtained in this way. This construction will be useful in later sections.
\end{example}

\subsection{The stack $B\mu_n$}
\label{subsec:mu_n}

\subsubsection{The group scheme $\mu_n$} 
\label{subsubsec:mu_n}
Let $n$ be a positive integer, then $\mu_n:=\spec \ZZ[t]/(t^n-1)$ is a group scheme with the group structure given by 
$$
\spec \ZZ[t]/(t^n-1)\times_{\spec\ZZ}\spec \ZZ[t]/(t^n-1)\to \spec \ZZ[t]/(t^n-1), \qquad (a,b)\mapsto ab,
$$
which is induced by the ring homomorphism
$$
\varphi: \ZZ[t]/(t^n-1) \to \ZZ[u]/(u^n-1)\otimes_{\ZZ}\ZZ[v]/(v^n-1), \qquad t\mapsto u\otimes v.
$$
Indeed, for every field $K$, given two points with coordinates $a,b\in\mu_n(K)$, the ideal of the point $(a,b)\in\mu_{n,K}\times_{\spec K}\mu_{n,K}$ is $$I_{(a,b)}=((u-a)\otimes 1,1\otimes(v-b))\subseteq K[u]/(u^n-1)\otimes_KK[v]/(v^n-1).$$
We observe that 
$$\varphi(t-ab)=u\otimes v-ab\otimes 1=((u-a)\otimes 1)(1\otimes(v-b)) + (u-a)\otimes b + a\otimes(v-b)\in I_{(a,b)}.$$
Thus $t-ab\in\varphi^{-1}(I_{(a,b)})$ and since $(t-ab)$ is a maximal ideal in $K[t]/(t^n-a)$ we have $\varphi^{-1}(I_{(a,b)})=(t-ab)$. This shows that $\varphi$ is the ring homomorphism that defines the group structure on $\mu_n$.

\subsubsection{Rational points of $B\mu_n$} 
In this section, for any non-zero element \(a\) in a field \(K\), we will denote by \(a^{1/n} \in \ol{K}^{\times}\) a choice of an \(n\)-th root of \(a\). The following theorem describes the isomorphism classes of \(K\)-points on \(B\mu_n.\)

\begin{theorem}\label{thm:mu_n}
Let $K$ be a field with a fixed algebraic closure $\overline K$. Assume that $n$ is invertible in $K$, and let $\zeta_n$ be a primitive $n$-th root of $1$ in $\overline K$. Then there are bijections between the following objects:

\begin{enumerate}
\item the set $B\mu_n(K)$, and
\item the set $K^{\times}/(K^{\times})^{n}=H^1_{\textit{\'et}}(K,\mu_{n,K})$.
\end{enumerate}
The bijection is given by
$$
K^{\times}/(K^{\times})^{n} \to (B\mu_n)(K), \qquad [a] \mapsto [\spec K_{a^{-1}}], 
$$
where \(K_{a}:=K[t]/(t^n-a)\) for all $a\in K^\times$.
\end{theorem}
\begin{remark}Note that $K_{a}$ is a field if and only if the polynomial $t^n-a$ is irreducible. Otherwise, $K_{a}$ is just an \'etale $K$-algebra of dimension $n$, as $\mathrm{char}(K)\nmid n$. 
By \cite[Theorem VI.9.1]{Lang-Algebra} irreducibility holds if and only if $a\notin K^p$ for any prime number $p\mid n$ and $a\notin -4k^4$ whenever $4\mid n$.
\end{remark}

\begin{proof}
As in \S \ref{subsec:BG} and Notation \ref{notation:BG-points}, $B\mu_{n}(K)$ denotes the set of isomorphism classes of torsors under $\mu_{n}$
over $K$. By  \cite[Proposition 4.5.6]{Olsson-Book} and \cite[Corollary III.4.7]{MR559531} (see also \cite[(2.11)]{skorob-torsors}), such isomorphism classes are classified by the flat cohomology group $H^1_{\textit{fl}}(K,\mu_{n,K})$. Since $n$ is coprime to the characteristic of $K$, the group scheme $\mu_{n,K}$ is smooth 
over $K$, and flat cohomology can be replaced by \'etale cohomology by \cite[Theorem III.3.9]{MR559531}. Hence, there is a bijection between $(B\mu_n)(K)$ and $H^1_{\textit{\'et}}(K,\mu_{n,K})$.

    We prove that $H^1_{\textit{\'et}}(K,\mu_{n,K})\cong K^\times/(K^\times)^n$ via étale cohomology on the Kummer exact sequence
\begin{equation}\label{eq:Kummer exact sequence}
1\to\mu_n(\overline K)\to \overline K^\times \stackrel{x \mapsto x^n}{\to} \overline K^\times \to 1.
\end{equation}
Since $K$ is a field, $H^1_{\textit{\'et}}(K,\mu_{n,K})$ can be computed as the Galois cohomology group $H^1(\Gal{K}, \mu_n(\overline K))$, where $\Gal{K}$ is the absolute Galois group of $K$.
Taking Galois cohomology of \eqref{eq:Kummer exact sequence} gives the long exact sequence
\[
1 \to H^0(\Gal{K}, \mu_n(\overline K)) \to H^0(\Gal{K}, \overline{K}^{\times}) \to H^0(\Gal{K}, \overline{K}^{\times}) \xrightarrow{\delta} H^1(\Gal{K}, \mu_n(\overline K)) \to H^1(\Gal{K}, \overline{K}^{\times}) \to \cdots
\]
Since $H^1(\Gal{K}, \overline{K}^{\times})=0$ by Hilbert 90 \cite[Proposition III.4.9]{MR559531}, we get the exact sequence
\[
1 \to \mu_n(\overline K) \cap K \to K^{\times} \stackrel{x \mapsto x^n}{\to} K^{\times} \xrightarrow{\delta} H^1(\Gal{K}, \mu_n(\overline K)) \to 0,
\]
which shows that 
$K^{\times}/(K^{\times})^n \cong H^1(\Gal{K}, \mu_n(\overline K))$, via the boundary map, \(\delta\). 

Consider the function
$$
K^\times/(K^\times)^n \to \{K_a:a\in K^{\times}\}/\sim, \qquad [a]\mapsto [K_{a^{-1}}],
$$
where $\sim$ indicates isomorphisms of $K$-algebras.
We will show that for any \(a \in K^{\times}\), $K_{a^{-1}}$ is a $\mu_{n,K}$-torsor over $K$ corresponding to the image $\delta([a]) \in H^1_{\textit{\'et}}(K,\mu_{n,K})$, thus showing that the diagram 
\begin{equation*}
    \begin{tikzcd}
    & B\mu_n(K) = \left\{ \substack{ \text{isomorphism classes } \\ \text{ of }\mu_n \text{ torsors over } K 
    }
    \right\} & \\
    &{ }   & \\
    K^{\times}/(K^{\times})^n  
    \arrow[rr, "\delta", "\sim" {swap}] \arrow[uur, "{[a]} \mapsto \spec K_{a^{-1}} " {anchor=south, rotate=28} ] & & H^1_{\textit{\'et}}(K,\mu_{n,K})  =\left\{ \substack{\text{Galois cohomology}\\ \text{classes}} \right\}  \arrow[uul, "\sim" {anchor=north, rotate=150}] 
    \end{tikzcd}
\end{equation*}
commutes and that the map \(a \mapsto K_{a^{-1}}\) is indeed a bijection.

Fix $\alpha\in\overline K$ such that $\alpha^n=a$. 
For every $r\in\{0,1,\dots,n-1\}$, there is a cocycle
$$\delta_r:\Gamma_K\to\mu_n(\overline K), \qquad \sigma\mapsto \sigma(\zeta_n^r\alpha)(\zeta_n^r\alpha)^{-1},$$ 
whose image in \(H^1(\Gal{K}, \mu_{n,K})\) agrees with \(\delta([a])\) by \cite[Remark 3.2.4]{Gille-Szamueli} and is independent of \(r\) since different choices of \(r\) yield cohomologous cocycles. 
Further, \(\delta_r\) induces a twisted action of \(\Gal{K}\) on \(\ol{K}[t]/(t^n -1)\) as follows: an element $\sigma \in \Gal{K}$ acts by the $K$-automorphism 
$\delta_{r,\sigma}$ defined by $\delta_{r,\sigma}(\beta)=\sigma(\beta)$ for all $\beta\in\overline K$ and $\delta_{r,\sigma}(t)=\delta_r(\sigma)t$.
Consider the isomorphism of $\overline K$-algebras 
\begin{equation}\label{eq:isom_mu_n}
\psi_r:\overline K[t]/(t^n-1)\to\overline K[t]/(t^n-a^{-1}), \qquad t\to \zeta_n^r\alpha t.\end{equation}
Then the diagram
\begin{equation*}
\xymatrix{
\overline K[t]/(t^n-1) \ar[d]^{\delta_{r,\sigma}} \ar[r]^{\psi_r} & \overline K[t]/(t^n-a^{-1}) \ar[d]^{\sigma}\\
\overline K[t]/(t^n-1) \ar[r]^{\psi_r} & \overline K[t]/(t^n-a^{-1})
},
\end{equation*}
commutes. Thus, the ring of invariants of the twisted Galois action corresponding to \(\delta_r\) is given by  \(K_{a^{-1}} = K[t]/(t^n - a^{-1})\), which is the required $K$-form of $\mu_{n,K}$ (cf.~\cite[\S III.1.3]{serre-gal-coh} or Example \ref{example:twisting-by-Galois}).
On the level of group schemes, $\spec K_{a^{-1}}$ is a $\mu_n$-torsor over $K$ with the torsor structure given explicitly by 
\begin{equation*}
\label{eq:isom mu_n}
K[t]/(t^n-a^{-1})\to K[t]/(t^n-a^{-1})\otimes_K K[u]/(u^n-1),\qquad t\mapsto t\otimes u. \qedhere
\end{equation*}

\end{proof}

\begin{remark}
    Let $\zeta_n\in\overline K$ be a fixed $n$-th root of $1$. For every $a\in K$, the extension $K(\zeta_n)\subseteq K(a^{\frac 1n},\zeta_n)$ is Galois, and hence uniquely determined by $a$, rather than by the choice of an $n$-th root of $a$. We denote by $(\star)$ the set of subfields of $\overline K$ of the form $K(a^{1/n}, \zeta_n)$ together with the $\mu_n(\overline K)$-action that sends $a^{\frac 1n}$ to $\zeta_n^{\frac nd} a^{\frac 1n}$, where $d=[K(a^{1/n}, \zeta_n):K(\zeta_n)]$.
    We observe that there is a surjective function 
    \begin{equation}\label{eq:identification}
K^{\times}/(K^{\times})^{n} \to (\star), \qquad  [a] \mapsto K(a^{1/n}, \zeta_n).
\end{equation}
 The Galois group $\gal(K(a^{\frac 1n},\zeta_n)/K(\zeta_n))$ acts on the roots of $a$ by multiplication by \(d\)-th roots of unity
where $d=[K(a^{\frac 1n},\zeta_n):K(\zeta_n)]$. 
Let $a,b\in K^\times$ such that \(K(a^{1/n}, \zeta_n) = K(b^{1/n}, \zeta_n)\). Let $\alpha,\beta\in \overline K$ such that $\alpha^n=a$ and $\beta^n=b$ and the \(\mu_n(\ol{K})\) action is given by \(\zeta_n \cdot \alpha = \zeta_n^{n/d}\alpha\) and \(\zeta_n \cdot \beta = \zeta_n^{n/d}\beta\). Let $d=[K(\alpha,\zeta_n):K(\zeta_n)]$.
Since $d\mid n$ and $\alpha^d, \beta^d\in K(\zeta_n)$ by \cite[Theorem VI.6.2]{Lang-Algebra}, $\beta,\beta^2,\dots,\beta^{d}$ is a basis of $K(\beta,\zeta_n)$ over $K(\zeta_n)$. 
Write $\alpha=\sum_{i=1}^{d}b_i\beta^i$ with $b_1,\dots,b_d\in K(\zeta_n)$. Since each element of $\gal(K(a^{\frac 1n},\zeta_n)/K(\zeta_n))$ acts on both $\alpha$ and $\beta$ by multiplication by the same power of $\zeta_n^{\frac nd}$, we get that $b_2=\dots=b_d=0$ and $b_1\neq 0$.
Let $\gamma=b_1$, then $\gamma\in K(\zeta_n)^\times$ and $\gamma^n=a/b\in K.$ 
We conclude that the function \eqref{eq:identification} factors through $K(\zeta_n)^{\times}/(K(\zeta_n)^{\times})^{n}$ and that \eqref{eq:identification} is injective if and only if the group homomorphism 
\begin{equation}\label{eq:K/K^n}
K^{\times}/(K^{\times})^{n}\to K(\zeta_n)^{\times}/(K(\zeta_n)^{\times})^{n}
\end{equation}
induced by the inclusion $K\subseteq K(\zeta_n)$ is injective. Indeed, the injectivity of the induced map from $K(\zeta_n)^{\times}/(K(\zeta_n)^{\times})^{n}$ to $(\star)$ follows from Kummer Theory (see for instance \cite[Corollary IV.3.6]{neukirch}). If $K=\mathbb Q$ and $n=3$, the homomorphism \eqref{eq:K/K^n} is injective, as $(c+d\zeta_3)^3\in\mathbb Q$ if and only if $cd=0$ or $c=d$, and in the latter case, we get $(c+c\zeta_3)^3=(-c)^3$. If $K=\mathbb Q$ and $n=4$, the homomorphism \eqref{eq:K/K^n} is not injective, as $-4=(1+\zeta_4)^4$ is a $4$-th power in $\mathbb Q(\zeta_4)$ but not in $\mathbb Q$. More generally, since the map \eqref{eq:K/K^n} is the restriction map on cohomology, its kernel is \(H^1(\gal(K(\zeta_n)/K), \mu_n)\), by the inflation restriction sequence (\cite[Corollary 2.4.2]{neukirch-coh-number-fields}).
\end{remark}

\begin{remark}
    Note that $\spec K(a^{\frac 1n},\zeta_n)\to\spec K$ does not need to be a $\mu_n$-torsor if $\zeta_n\notin K$, because a trivialization is given by $\overline K\otimes_K K(a^{\frac 1n},\zeta_n)\cong (\overline K)^{dr}$ where $d=[K(a^{\frac 1n},\zeta_n):K(\zeta_n)]$ and $r=[K(\zeta_n):K]$, while the trivial torsor is given by $\overline K[t]/(t^n-1)\cong (\overline K)^n$.
\end{remark}

\begin{remark}
    We refer the more categorically inclined reader to \cite[Example 4.5.8]{Olsson-Book} for a description of the category of \(\mu_n\)- torsors over a scheme \(S\) in which \(n\) is invertible. 
\end{remark}

\subsection{The stack $B(\ZZ/n\ZZ)$} 
\label{subsec:Z/nZ}

\subsubsection{The constant group scheme $\underline {\ZZ/n\ZZ}$}
Given a positive integer $n$, we denote by $B(\ZZ/n\ZZ)$ the classifying stack of the constant group scheme $\underline {\ZZ/n\ZZ}$ defined by the finite cyclic group $\ZZ/n\ZZ$. Then $\underline{\ZZ/n\ZZ}=\spec\ZZ^n$, where $\ZZ^n$ is the free $\ZZ$-module of rank $n$ with a fixed basis $\{e_i:i\in\ZZ/n\ZZ\}$ and ring structure given by componentwise addition and multiplication, i.e.,  $$\sum_{i\in\ZZ/n\ZZ}a_ie_i+\sum_{i\in\ZZ/n\ZZ}b_ie_i=\sum_{i\in\ZZ/n\ZZ}(a_i+b_i)e_i, \qquad \left(\sum_{i\in\ZZ/n\ZZ}a_ie_i\right)\left(\sum_{i\in\ZZ/n\ZZ}b_ie_i\right)=\sum_{i\in\ZZ/n\ZZ}(a_ib_i)e_i,$$
for all $a_i,b_i\in\ZZ$ and $i\in\ZZ/n\ZZ$. We observe that the ring structure on $\ZZ^n\otimes_{\ZZ}\ZZ^n\cong\ZZ^{n^2}$ is given by componentwise addition and multiplication for the basis $\{e_i\otimes e_j:i,j\in\ZZ/n\ZZ\}$.

\begin{lemma}
    The group scheme structure on $\underline{\ZZ/n\ZZ}$ is induced by 
    $$\varphi:\ZZ^n\to\ZZ^n\otimes_{\ZZ}\ZZ^n, \qquad \sum_{i\in\ZZ/n\ZZ}a_ie_i\mapsto \sum_{i,j\in\ZZ/n\ZZ}a_{i+j}e_i\otimes e_j.$$
\end{lemma}
\begin{proof}
    If $K$ is a field then $\underline{\ZZ/n\ZZ}(K) \cong \ZZ/n\ZZ$. If we denote by $P_\alpha \in \underline{\ZZ/n\ZZ}(K)$ the preimage under this map of $\alpha\in\ZZ/n\ZZ$, the group scheme structure is given by 
    $$\underline{\ZZ/n\ZZ}(K)\times \underline{\ZZ/n\ZZ}(K)\to \underline{\ZZ/n\ZZ}(K),\qquad (P_\alpha,P_\beta)\mapsto P_{\alpha+\beta}.$$
    For $i,j\in\ZZ/n\ZZ$, let $c_{i,j}=0$ if $i=j$ and $c_{i,j}=1$ if $i\neq j$.
    In the chosen notation, for $\alpha\in\ZZ/n\ZZ$ the ideal of the point $P_{\alpha}$ is $\{\sum_{i\in\ZZ/n\ZZ}a_ie_i:a_{\alpha}=0\}\subseteq K^n$ and it is generated by $\sum_{i\in\ZZ/n\ZZ}c_{i,\alpha}e_i$. Fix $\alpha,\beta\in\ZZ/n\ZZ$, then the ideal of $(P_\alpha,P_\beta)\in \underline{\ZZ/n\ZZ}(K)\times \underline{\ZZ/n\ZZ}(K)$ is the ideal generated by the two elements $v_\alpha:=\sum_{i,j\in\ZZ/n\ZZ}c_{i,\alpha}e_i\otimes e_j$ and  $v_\beta:=\sum_{i,j\in\ZZ/n\ZZ}c_{j,\beta}e_i\otimes e_j$. We observe that $\varphi\left(\sum_{i\in\ZZ/n\ZZ}c_{i,\alpha+\beta}e_i\right)= \sum_{i,j\in\ZZ/n\ZZ}c_{i+j,\alpha+\beta}e_i\otimes e_j$ belongs to the ideal generated by $v_\alpha$ and $v_\beta$. Indeed,
    $$v_\alpha+v_\beta-v_\alpha v_\beta=\sum_{i,j\in\ZZ/n\ZZ}(c_{i,\alpha}+c_{j,\beta}-c_{i,\alpha}c_{j,\beta})e_i\otimes e_j,$$
    as $c_{i+j,\alpha+\beta}=0$ if $i=\alpha$ and $j=\beta$ and $c_{i,\alpha}+c_{j,\beta}-c_{i,\alpha}c_{j,\beta}=0$ if and only if $i=\alpha$ and $j=\beta$.
    Since the ideal generated by $\sum_{i\in\ZZ/n\ZZ}c_{i,\alpha+\beta}e_i$ is a maximal ideal of $K^n$, we conclude that it is equal to $\varphi^{-1}((v_\alpha,v_\beta))$.
\end{proof}

\subsubsection{$\underline {\ZZ/n\ZZ}$-torsors}

Since $\underline{\ZZ/n\ZZ}_K =\spec(K^n)$ is finite and \'etale over $K$, so is every $\underline{\ZZ/n\ZZ}$-torsor over $K$ (see \cite[Proposition III.4.2]{MR559531}). Thus, torsors under $\underline{\ZZ/n\ZZ}$ are Galois coverings of $\spec K$ with Galois group $\ZZ/n\ZZ$ in the sense of \cite[Remark I.5.4]{MR559531}. In the following theorem, we give some equivalent ways of describing this set.

\begin{theorem}\label{thm:Z/nZ}
Let $K$ be a field of characteristic zero with a fixed algebraic closure $\overline K$. Then there is a bijection between the following objects:
\begin{enumerate}
\item the set of $K$-points of the stack $B(\ZZ/n\ZZ)$;
\item the set of classes of continuous $1$-cocycles
$\Hom_{cts}(\Gal{K}, \ZZ/n\ZZ)$, where $\Gal{K}$ is equipped with the profinite topology
and $\ZZ/n\ZZ$  is a discrete topological space;
\item the set of isomorphism classes of finite \'etale \(\ZZ/n\ZZ\)-algebras $L/K$ .
\end{enumerate}
\end{theorem}
\begin{proof}
As in the proof of Theorem \ref{thm:mu_n}, since $\underline{\ZZ/n\ZZ}$ is smooth over $K$, the isomorphism classes of torsors under $\underline{\ZZ/n\ZZ}$ over $K$ are classified by $H^1_{\textit{\'et}}(K,\underline{\ZZ/n\ZZ})=H^1(\Gal{K},\ZZ/n\ZZ)$. Since the $\Gal{K}$-module structure on $\ZZ/n\ZZ$ is given by the trivial action of $\Gal{K}$, the 1-cocycles are just continuous group homomorphisms $\Gal{K}\to \ZZ/n\ZZ$ and the coboundaries are trivial (cf.~\cite[\S I.5.1]{serre-gal-coh}). Thus $H^1(\Gal{K},\ZZ/n\ZZ)=\Hom_{cts}(\Gal{K},\ZZ/n\ZZ)$ is in bijection with $B(\ZZ/n\ZZ)(K)$. This establishes the bijection between (1) and (2).

The map between (2) and (3) is given by
\begin{align*}
   \left(\varphi : \Gal{K} \to \ZZ/n\ZZ \right) \mapsto \prod_{r\in \coker\varphi}\ol{K}^{\ker\varphi}.
\end{align*}
Suppose the image of \(\varphi\) is isomorphic to \(m\ZZ/n\ZZ\) for some \(m \in \ZZ\), whence its cokernel is isomorphic to \(\ZZ/m\ZZ\). Then $K_\varphi:=\overline K^{\ker\varphi}$ is Galois over \(K\) with Galois group $m\ZZ/n\ZZ$, and this action on \(K_{\varphi}\) together with the permutation action of $\ZZ/m\ZZ$ on the factors induces a \(\ZZ/n\ZZ\)-action on the \'etale algebra \(L:= \prod_{r \in \ZZ/m\ZZ} K_{\varphi} \),
thus giving it the structure of an \'etale \(\ZZ/n\ZZ\)-algebra. If \(\varphi \neq \varphi'\) such that \(\ker \varphi = \ker \varphi'\), then $K_\varphi=K_{\varphi'}$ as subsets of $\overline K$, and the images \(\im(\varphi)\) and \(\im(\varphi')\) are isomorphic to \(m\ZZ/n\ZZ\) for some \(m \in \ZZ\). However, the induced actions of $\im(\varphi)$ and $\im(\varphi')$ on \(L\) are different and correspond to picking different generators of \(m\ZZ/n\ZZ\). Thus, they give different maps \(\ZZ/n\ZZ \hookrightarrow \aut_K(L)\), establishing that the map between (2) and (3) is an injection. The surjectivity of the map follows from the definition of a \(\ZZ/n\ZZ\)-algebra: since \(\ZZ/n\ZZ\) must act transitively on the idempotents of the \'etale algebra, it must be of the form \(\prod_{r \in \ZZ/m\ZZ} \ol{K}^{H}\) for some \(m|n\) and \(H\) a subgroup of \(\Gal{K}\) of index $n/m$. The Galois correspondence \cite[Theorem VI.1.1, \S VI.14]{Lang-Algebra} then finishes the proof. This proof is a special case of that of \cite[Lemma 2.6]{wood-probabilities}.

We now describe explicitly the map between (3) and (1).
Denote by \(\{e_i\}_{i\in\ZZ/n\ZZ}\) the standard basis vectors of \(\overline{K}^n \cong \overline{K} \otimes_{\ZZ} \ZZ^n\). Given \(\varphi \in \Hom_{cts}(\Gal{K}, \ZZ/n\ZZ)\), let \(K_{\varphi} = \ol{K}^{\ker \varphi}\) and write $[K_\varphi:K]=n/m$ as before. The \'etale algebra $\prod_{r\in\ZZ/m\ZZ}K_{\varphi}$ is the ring of invariant elements of $\overline K^n$ under the twisted Galois action 
$$
\Gal{K}\times\overline K^n\to\overline K^n,\qquad (\sigma,v)\mapsto \varphi(\sigma)(\sigma(v)),
$$
where \(\varphi(\sigma)\) acts on $\overline{K}^n$ by sending \(e_i \mapsto e_{i+\varphi(\sigma)}\). This describes the \'etale algebra as a twist of the trivial torsor in the sense of Example \ref{example:twisting-by-Galois}. In order to describe the action of the group  {scheme} \(\underline{\ZZ/n\ZZ}\),
let $[\cdot]:\ZZ/n\ZZ\to\ZZ/m\ZZ$ be the reduction modulo $m$. The elements of $\prod_{r\in\ZZ/m\ZZ}K_{\varphi}$ can be written uniquely as $\sum_{r\in\ZZ/m\ZZ}v_r$, where $v_r\in K_\varphi$ and $v_r=\sum_{i\in\ZZ/n\ZZ, [i]=r}a_{i}e_i$ as an element of $\overline K^n$. For every $j\in\ZZ/n\ZZ$, the shift $v_r(j)=\sum_{i\in\ZZ/n\ZZ,[i]=r}a_{i+j}e_i$ is an element of $K_\varphi$.
Then $\prod_{r\in\ZZ/m\ZZ}K_{\varphi}$ defines a torsor under $\ZZ/n\ZZ$ with the following torsor structure:
\begin{equation*}
\prod_{r\in\ZZ/m\ZZ}K_{\varphi}\to \prod_{r\in\ZZ/m\ZZ}K_{\varphi}\otimes K^n,\qquad \sum_{r\in\ZZ/m\ZZ}v_r\mapsto \sum_{r\in\ZZ/m\ZZ}\sum_{j\in\ZZ/n\ZZ}v_{r}(j)\otimes e_j. \qedhere
\end{equation*}
\end{proof}

\subsection{Comparison between rational points on $B\mu_n$ and $B\ZZ/n\ZZ$} 
\label{subsec:comparison}
Let $K$ be a field of characteristic zero with a fixed algebraic closure $\overline K$. Let $\zeta_n\in\overline K$ be a primitive $n$-th root of 1.

\begin{lemma}
\label{rmk:mu_n isom Z/nZ}
If $\zeta_n\in K$, then $B\mu_n\cong B(\ZZ/n\ZZ)$ over $K$.
\end{lemma}

\begin{proof}
We observe that if $\zeta_n\in K$, then there is an isomorphism of group schemes $\underline{\ZZ/n\ZZ}\to\mu_{n,K}$ given by
$$
K[t]/(t^n-1)\to K^n,\qquad t^s\mapsto \sum_{i\in\ZZ/n\ZZ}\zeta_n^{si}e_i,
$$
with inverse
\begin{equation}\label{eq:isom comparison}
K^n\to K[t]/(t^n-1),\qquad \sum_{i\in\ZZ/n\ZZ}a_ie_i\mapsto \frac 1 n \sum_{s=0}^{n-1}\sum_{i\in\ZZ/n\ZZ}a_i\zeta_n^{-si}t^s.
\end{equation}
Indeed, using the notation in \S \ref{subsec:Z/nZ}, the isomorphism is compatible with the group structures as
$$
\frac 1 n \sum_{r=0}^{n-1}\sum_{i\in\ZZ/n\ZZ} a_i \zeta_n^{-ri}(u\otimes v)^r
= \frac 1 {n^2} \sum_{r=0}^{n-1}\sum_{s=0}^{n-1}\sum_{i,j\in\ZZ/n\ZZ}a_{i+j}\zeta_n^{-ri-sj}u^r\otimes v^s
$$
holds in $K[u]/(u^n-1)\otimes_KK[v]/(v^n-1)$ for all choices of $a_i\in K$ thanks to the fact that $\sum_{j\in\ZZ/n\ZZ}\zeta_n^{(r-s)j}$ equals $n$ if $r=s$ and vanishes if $r\neq s$. Hence, the statement of the lemma holds. In particular, they have the same sets of $K$-points.
\end{proof}

If $\zeta_n\notin K$, then no connected $\mu_n$-torsor over $K$ is a $\underline{\ZZ/n\ZZ}$-torsor over $K$ and no connected $\underline{\ZZ/n\ZZ}$-torsor over $K$ is a $\mu_n$-torsor over $K$. Indeed, connected torsors are fields, and  field extensions $L/K$ that are torsors under $\mu_{n}$ are not Galois, while field extensions $L/K$ that are torsors under $\underline{\ZZ/n\ZZ}$ are always Galois extensions, as follows from the descriptions in Theorem \ref{thm:mu_n} and Theorem \ref{thm:Z/nZ}.

    We now spell out the difference between the classifying groups of $K$-forms and torsors for $\mu_{n,K}$ and $\underline{\ZZ/n\ZZ}$ over a field $K$ of characteristic $0$.
    By \cite[\S III.1.3]{serre-gal-coh}, the $K$-forms of $\mu_{n,K}$ are classified by the cohomology group $H^1(\Gal{K},
    \mathrm {Aut}(\overline K[t]/(t^n-1))
        )$, where 
    the automorphism group $\mathrm {Aut}(\overline K[t]/(t^n-1))$ (in the category of $\overline K$-algebras) 
    is the symmetric group $S_n$ on the set of $n$ elements, and the $\Gal{K}$-module structure is given by conjugation of automorphisms for the $\Gal{K}$-action on $\mu_{n}(\overline K)$ induced by the base change from $K$ to $\overline K$.
    The $K$-forms of $\underline{\ZZ/n\ZZ}$ are classified by the cohomology group $H^1(\Gal{K},
    \mathrm {Aut}(\overline K^n)
    )$, where the automorphism group $\mathrm {Aut}(\overline K^n)$ (in the category of $\overline K$-algebras)
    is the symmetric group $S_n$
    with $\Gal{K}$-module structure given by conjugation of automorphisms for the $\Gal{K}$-action on $\underline{\ZZ/n\ZZ}$ induced by the base change from $K$ to $\overline K$. Notice that in this second case the $\Gal{K}$-module structure on $S_n$ is trivial.

    Since $\mu_{n,\overline K}\cong\underline{\ZZ/n\ZZ}$ over $\overline K$, the two groups have the same set of $K$-forms. We identify $\ZZ/n\ZZ$ as the subgroup of $S_n$ induced by the group operation on $\ZZ/n\ZZ$, where an element $r\in\ZZ/n\ZZ$ is identified with the permutation
    \begin{equation}\label{eq:subgroup}
    r: \ZZ/n\ZZ \to \ZZ/n\ZZ, \quad i\mapsto r+i.
    \end{equation}

    We observe that the $\mu_{n,K}$-torsors are classified by the cohomology of the subgroup $\mu_{n}(\overline K)\cong\ZZ/n\ZZ < S_n$, as the $\Gal{K}$-module structure on $\ZZ/n\ZZ$ induced by 
    $\aut(\overline K[t]/(t^n-1))$
    coincides with the $\Gal{K}$-module structure on $\mu_n(\overline K)$ induced by the base change from $K$ to $\overline K$. Similarly, the $\underline{\ZZ/n\ZZ}$-torsors over $K$ are classified by the cohomology of the subgroup $\ZZ/n\ZZ < S_n$  with trivial $\Gal{K}$-module structure.

\begin{example}
    If $\zeta_n\notin K$, then some non-connected $\mu_{n,K}$-torsors are not $\underline{\ZZ/n\ZZ}$-torsors and vice versa. For example, $\mu_{n,K}$ is not a $\ZZ/n\ZZ$-torsor as the polynomial $t^n-1$ factors over $K$ as a product of irreducible polynomials of different degrees, where at least one has degree 1 and another one is the minimal polynomial of $\zeta_n$, which has degree $>1$.

    On the other hand, some non-connected $\mu_{n,K}$-torsors are $\underline{\ZZ/n\ZZ}$-torsors. For example, if $K=\mathbb Q(\zeta_8)$, $n=16$ and $a=-1$, the $K$-algebra $K_{-1}=K[t]/(t^{16}+1)$ is both a $\mu_{16}$-torsor and a $\underline{\ZZ/16\ZZ}$-torsor. Indeed, it is clearly a $\mu_{16}$-torsor by Theorem~\ref{thm:mu_n}. To show that it is a $\underline{\Z/16\Z}$-torsor, let $A_i=K[t]/(t^4-\zeta_8^i)$ for $1\leq i\leq 8$, and note that there are isomorphisms of $K$-algebras
    $$
    K[t]/(t^{16}+1)\cong A_1\times A_3\times A_5\times A_7 \cong A_1^4,
    $$
    where the first isomorphism is given by factoring the polynomial $t^{16}+1$ over $K$, and the second isomorphism is given by 
    $$
    (t_1,t_3,t_5,t_7)\mapsto (t_1, t_3^3, \zeta_8t_5, \zeta_8 t_7^3)
    $$
    with inverse
    $$
    (u_1,u_3,u_5,u_7)\mapsto (u_1, u_3^{11}, \zeta_8^{-1} u_5, \zeta_8^{-3} u_7^{11}).
    $$
    The $K$-algebra $A_1$ is a $\underline{\ZZ/4\ZZ}$-torsor as in Theorem \ref{thm:Z/nZ}. We show that  $A_1^4$ is a $\underline{\ZZ/16\ZZ}$-torsor. Consider the isomorphism
    $$
    \varphi: \prod_{i=1}^4\overline K^4 \to \prod_{i=1}^4\overline K[t]/(t^4-\zeta_8), \quad \sum_{i=1}^4\sum_{j\in\ZZ/4\ZZ} a_{i,j} e_{i,j} \mapsto \left(\frac 14\sum_{s=0}^{3}\sum_{j\in\ZZ/4\ZZ} a_{i,j} \zeta_4^{-js}\zeta_{32}^{-s} t^s\right)_{1\leq i\leq 4}
    $$
    obtained by combining the isomorphisms \eqref{eq:isom comparison} with \eqref{eq:isom_mu_n} for $r=0$ on each of the four components. Here for each $i\in\{1,\dots,4\}$, $e_{i,1},\dots,e_{i,4}$ denotes a fixed basis of $\overline K^4$.
    Then $\varphi$
    induces the cocycle
$$\widetilde\varphi: \Gal{K}\to\aut(\overline K^{16})=S_{16}, \quad \sigma\mapsto \varphi^{-1}\circ \sigma\circ\varphi\circ\sigma^{-1},$$ given by 
$$
\widetilde \varphi(\sigma): \prod_{i=1}^4\overline K^4\to\prod_{i=1}^4\overline K^4,\quad \sum_{i=1}^4\sum_{j\in\ZZ/4\ZZ}a_{i,j}e_{i,j} \mapsto \sum_{i=1}^4\sum_{j\in\ZZ/4\ZZ} a_{i,j}e_{i, j+v_\sigma}
$$
for $\sigma\in \Gal{K}$ and $v_\sigma\in\ZZ/4\ZZ$ such that $\sigma(\zeta_{32})=\zeta_4^{v_\sigma}\zeta_{32}$.
Using the exact sequence
$$
0\longrightarrow\ZZ/4\ZZ\stackrel{4\cdot}{\longrightarrow}\ZZ/16\ZZ\longrightarrow\ZZ/4\ZZ\longrightarrow 0
$$
and \eqref{eq:subgroup},
we conclude that the cocycle $\widetilde \varphi$ is the continuous homomorphism $$\Gal{K}\to\ZZ/16\ZZ, \quad \sigma\mapsto 4v_\sigma.$$
\end{example}

\section{Heights on classifying stacks}
\label{sec:heights}
As described in the previous sections, \'etale algebras with specified Galois groups are in bijection with points on certain stacks. Further, since Hermite's theorem \cite[\S III.2]{neukirch}
implies that the number of such number fields with bounded discriminant is finite, the discriminant can be thought of as a  {height function} satisfying the Northcott property on the stack $BG$. This idea has been fleshed out in much greater detail in recent work \cite{ESZB}, \cite{dardathesis}, \cite{darda-yasuda22}, building on previous work of Yasuda \cite{yasuda-motivic} and Yasuda--Wood \cite{yasuda-wood}.

In this section we collect a number of approaches to defining heights of points on $BG$. In \S \ref{subsec:discriminant}, we compute the discriminant of the \'etale algebras \(K[t]/(t^n-a)\). In \S\S \ref{subsec:eszb-height}, \ref{subsec:darda-height} and \ref{subsec:darda-yasuda-heights}, we describe heights on \(BG\) following the framework of \cite{ESZB}, \cite{dardathesis} and \cite{darda-yasuda22} respectively.

\subsection{Discriminants of the points on $B\mu_n$}
\label{subsec:discriminant}

For $a\in K^\times$, we compute the discriminant $\Delta(K_a/K)$ of the \'etale algebras $K_a=K[t]/(t^n-a)$ following \cite[\S9]{dardathesis}.
\begin{proposition}
\label{prop:discriminant-Bmun}
    Let $K$ be a field such that $n$ is invertible in $K$. Let $a\in K^\times$, and $K_a=K[t]/(t^n-a)$.
    Then there is a positive constant $C_{K,n}$ independent of $a$ such that 
    $$\Delta(K_a/K)=\prod_{v\nmid n, v\mid a}\mathfrak p_v^{n-d_v}\prod_{v\mid n}\mathfrak p_v^{e(a,v)},$$
where $\mathfrak p_v$ is the prime ideal of $\OO_K$ corresponding to the finite place $v$, $d_v=\gcd(v(a),n)$, and $0\leq e(a,v)\leq C_{K,n}$.

If $K=\QQ$ and $a\in\ZZ$ is coprime to $n$, then $\Delta(K_a/K)=\prod_{p\mid a} p^{n-\gcd(v_p(a),n)}\prod_{p\mid n}p^{e(a,p)}$. If $a\in\ZZ$ is squarefree and coprime to $n$, then $K_a$ is a field and $\Delta(K_a/K)=a^{n-1}\prod_{p\mid n}p^{e(a,p)}$. 
\end{proposition}

Examples of complete computations of the discriminant can be found in  \cite[Proposition 2.7]{MR0718674} for  cyclotomic fields, and in \cite[Lemma 3]{Daberkow-Pohst} for the case where $n$ is prime and $\zeta_n\in K$. 

\begin{proof}
Let $v$ be a finite place of $K$ and $K_v$ the completion of $K$ at $v$ with ring of integers $\OO_v$ and uniformizer $\pi$. 
If $v\nmid n$,
$$\Delta(K_v(a^{\frac 1n})/K_v)
=\pi^{[K_v(a^{\frac 1n}):K_v](1-\frac{d_v}n)}\OO_v.$$
by \cite[Lemma 9.1.2.2]{dardathesis}, and as in \cite[Proposition 9.1.2.3]{dardathesis} we get $\Delta(K_v[t]/(t^n-a)/K_v)=\pi^{n-d_v}\OO_v$.
If $v\mid n$, then $v(\Delta(K_v(a^{\frac 1n})/K_v))$ is bounded in terms of $[K_v(a^{\frac 1n}):K_v]\leq n$ and $v$ (cf.~\cite[Theorem III.2.6]{neukirch}), and hence $e(a,n)=v(\Delta(K_v[t]/(t^n-a)/K_v))$ is bounded in terms of $[K_v(a^{\frac 1n}):K_v]\leq n$ and $v$.
By \cite[Corollary III.2.11]{neukirch} we get $\Delta(K_a/K)=\prod_{v\mid a, v\nmid n}\mathfrak p_v^{n-d_v}\prod_{v\mid n}\mathfrak p_v^{e(a,v)}$
where $\mathfrak p_v$ is the prime ideal of $\OO_K$ corresponding to the finite place $v$.

If $K=\QQ$ and $a\in\ZZ$ is squarefree and coprime to $n$, then $K_a$ is a field and $\Delta(K_a/K)=a^{n-1}\prod_{p\mid n}p^{e(a,p)}$. If $a\in\ZZ$ is coprime to $n$, then $\Delta(K_a/K)=\prod_{p\mid a} p^{n-\gcd(v_p(a),n)}\prod_{p\mid n}p^{e(a,p)}$. 
\end{proof}

\subsection{Heights on \(BG\)}

In this subsection we describe the two approaches taken in \cite{ESZB} and \cite{dardathesis} to talk about heights on \(BG\). 

\subsubsection{Heights with respect to a vector bundle}
\label{subsec:eszb-height}
If \(V\) is a projective variety over \(K\) with a very ample line bundle \(L\), one can define the Weil height with respect to \(L\)
by extending \(V\) and \(L\) to a model \(\ol{V}\) over \(\spec \OO_K\), pulling \(L\) back to a point \(\spec \OO_K \to \ol{V}\) and computing its ``degree''. 
See for instance \cite[Part B]{HindrySilverman}.
Such a height is geometric in nature, in that the asymptotic growth
of the number of points of bounded height reflect the geometry of \(V\). 
See for instance \cite{Batyrev-Manin}.
In \cite{ESZB}, the authors generalize this idea to stacks by defining heights with respect to vector bundles on (proper, Artin) stacks\footnote{We do not define these technical notions in this article. The stacks we are interested in do fall into this category. For definitions, we refer the reader to \cite{Olsson-Book} or \cite{alpernotes}.}. This approach unifies counting rational points on varieties and counting number fields under the same umbrella. In particular, on classifying stacks, the height in \cite{ESZB} specializes to some well known invariants used to count number fields, such as the discriminant.

\begin{definition}
    \label{definition:vector-bundle-quotient-stack}
    Let \(S\) be an affine scheme and \(G\) a 
    {smooth, 
    finite}
    group scheme over \(S\). A vector bundle on \([S/G]\) is the data of a vector bundle \(V\) on \(S\) along with a \(G\)-action that is {compatible}
    with the action of \(G\) on \(S\). If \(S = \spec K\) for a field \(K\) and \(G\) is constant, then a vector bundle on \(BG\) is a finite dimensional \(G\)-representation over \(K\). More generally, if \(G\) is a constant group scheme, and the action of \(G\) on \(S\) is trivial, 
     then a vector bundle on \([S/G] = BG_S\) 
    {corresponds to}
    a family of \(G\)-representations indexed by \(S\). 
\end{definition}
\begin{remark}
    There is a more general definition of vector bundles on stacks, which specializes to Definition \ref{definition:vector-bundle-quotient-stack} for quotient stacks. For the general definition, we refer the reader to \cite[\S 6.1]{alpernotes}, and for a proof that quasicoherent sheaves on the \'etale site of \(BG_S\) correspond to \(\OO_S\)-modules with left \(G\)-action, we refer the reader to \cite[Example 9.1.19]{Olsson-Book}.
\end{remark}

Let $\mathcal{X}$ be a nice\footnote{By a ``nice'' stack will mean a normal, proper, Artin stack with finite diagonal. The reader may just keep \(\mathcal{X} = BG\) in mind for the rest of this article.}
stack over the ring of integers $\mathcal{O}_K$ of a number field $K$, equipped with a vector bundle $\mathcal{V}$. A $K$-point of $\mathcal{X}$, $x : \spec K \to \mathcal{X}$ doesn't necessarily extend to an $\OO_K$ point of $\mathcal{X}$ due to failure of the valuative criterion of properness for stacks. However, it does extend to a $\mathcal{C}$-point $\bar x: \mathcal C \to \mathcal X$, where $\mathcal{C}$ is a stack over $\OO_K$ that is in some sense ``minimally ramified'' over $\OO_K$. This is called a tuning stack (\cite[\S 2.1]{ESZB}) and is birational to $\spec \OO_K$ (in fact, \(\mathcal{C} \to \spec \OO_K\) is a coarse space map, \cite[Definition 11.1.1]{Olsson-Book}).
More precisely, a tuning stack fits into a diagram of the form
\begin{equation}
\begin{tikzcd}
\spec K  \arrow[r] \arrow[rr, "x", bend left=30] \arrow[dr, hook]
& \mathcal{C} \arrow[d, "\pi"] \arrow[r, "\overline{x}"] & \mathcal{X}  \\
& \spec \OO_K & 
\end{tikzcd}
\label{fig:tuning-stack}
\end{equation}
where \(\pi\) is a birational coarse space map.
The authors of \cite{ESZB} use this diagram to define the height of \(x\) as follows. 

\begin{definition}[{\cite[Definition 2.11]{ESZB}}]
\label{defn:eszb-height}
    Let $\mathcal{V}$ be a vector bundle on $\mathcal{X}$. Then the height of a point
    $x : \spec K \to \mathcal{X}$ with respect to the vector bundle $\mathcal{V}$ is defined as:
    \[
     \height_{\mathcal{V}}(x) = - \deg(\pi_{*}\overline{x}^{*} (\mathcal{V}^{\vee})),
    \]
    where \(\deg\) refers to the Arakelov degree of the  determinant of $\mathcal V$ equipped with any metrization as in \cite[\S A3-A4]{ESZB}. The height is independent of the choice of the tuning stack $\mathcal C$.
\end{definition} 
\begin{remark}
\label{rmk:arakelov-degree}
 Definition \ref{defn:eszb-height} also extends to the case when \(K\) is the global function field of a smooth, proper curve \(C\) with \(\spec \OO_K\) replaced by \(C\). In that case, the degree in Definition \ref{defn:eszb-height} is the usual notion of the the degree of a divisor on the curve \(C\). 
    It is only when \(K\) is a number field that one needs to account for the infinite place and therefore consider the  {Arakelov degree}. For a quick introduction to Arakelov theory, we refer the reader to \cite{Arakelov74}. 
\end{remark}

The height in Definition \ref{defn:eszb-height} satisfies the Northcott property for appropriate choices of \(\mathcal{X}\) and \(\mathcal{V}\). In this article we only work with heights on \(BG\) which, as we will see in Propositions \ref{proposition:ESZB-discriminant} and \ref{prop:discriminant-Bmun}, agree with the discriminant at all but finitely many places, and in particular, satisfy the Northcott property.

\begin{remark}
\label{rmk:constant group}
 Let \(G\) be a finite group. 
 Let  \(\rho: G \hookrightarrow S_n\) be a degree \(n\) permutation representation of $G$. If \(x \in BG(K)\) is a connected torsor, then since \(G_K\) is a constant group scheme, via the identification \(BG(K) \cong H^1(\Gal{K}), G) \cong \Hom_{cts}(\Gal{K}, G) \),  \(x\) corresponds to a surjection \(x:\Gal{K} \to G\).
 When composed with \(\rho : G \to S_n\), this gives a homomorphism \(\rho_x : \Gal{K} \to S_n\). The composition thus defines an \'etale \(G\)-algebra \(\tilde{L} = (\ol{K})^{\ker \rho_x}\), which in particular is a Galois extension of $K$.
 Let \(H \subset S_n\) be the stabilizer of \(1 \in \{1, 2, \ldots , n \}\) in \(\rho(G)\).
 Then the fixed set \(\tilde{L}^{\rho^{-1}(H)}\) is an intermediate extension \(L/K\) with Galois closure \(\tilde{L}\). If the torsor
 \(x \in BG(K)\) is not connected, then the induced homomorphism \(\Gal{K} \to G\) may not be surjective. In this case, one can still generalize this construction by taking  \(\tilde{L} = \prod_{r \in \coker(x)}(\ol{K})^{\ker \rho_x}\) and \(L\) to be  the corresponding
 product of fixed fields of 1-point stabilizer subgroups. In what follows, we take \(L\) to be this \'etale \(G\)-algebra.
\end{remark}

\begin{proposition}[{\cite[Proposition 3.1]{ESZB}}]
\label{proposition:ESZB-discriminant}
    Let \(G\) be a finite group, and \(\underline{G}_{\OO_K}\) the associated constant group scheme over \(\spec \OO_K \).
    Let $\mathcal{V}$ be the vector bundle on $BG_{\spec \OO_K}$ corresponding to a degree $n$ permutation representation \(\rho: G \hookrightarrow S_n\) (equipped with the trivial metric described in \cite[\S A.2]{ESZB}). Let $x \in BG(K)$. The composition \(\Gal{K} \xrightarrow{x} G \xrightarrow{\rho} S_n\) defines a degree \(n\) \'etale algebra \(L/K\) as in Remark \ref{rmk:constant group}.
    Then
    \[
    \height_{\mathcal{V}}(x) = \frac{1}{2}\log|\Delta(L/K)|,
    \]
    where $|\Delta(L/K)|$ denotes the absolute norm of the discriminant of the extension \(L/K\). 
    \end{proposition}

In \cite{ESZB}, the authors prove Proposition \ref{proposition:ESZB-discriminant} by using a ``local decomposition'' of the height, and observing that the local factors are related to the Artin conductor of a certain representation, as in \cite[\S 4]{yasuda-wood}. We briefly outline the way to go from the setting in \cite{ESZB} to \cite{yasuda-wood}.

\begin{proof} 
    Let \(x \in BG(K)\) be a \(G\)-torsor over \(K\) and \(\rho : G \to S_n\) as in the statement of the proposition. One can extend \(\rho\) to a map \(G \to \GL_n(\OO_K)\) via permutation matrices. We will denote this by \(\rho\) as well.
 Let $V:=\mathcal{O}_K^{\oplus n}$ equipped with the trivial metric, let \(\mathcal{V} := ( V, \rho)\)
    be the vector bundle on \(BG_{\spec \OO_K}\) corresponding to \(\rho\) as in  Definition \ref{definition:vector-bundle-quotient-stack}, and let \(L\) and \(\tilde{L}\) be as in Remark \ref{rmk:constant group}.

    In the setup of 
    {\eqref{fig:tuning-stack},}
    let $\mathcal{X} = BG_{\spec \OO_K}$ and take $\mathcal{C} = [\spec \OO_{\tilde{L}}/G]$. 
    As a consequence, we may extend 
    {\eqref{fig:tuning-stack}}
    to the following  commutative diagram,
\begin{equation}    
    \centering
    \begin{tikzcd}
&\spec \OO_{\tilde{L}} \arrow[d, "p"] \arrow[r, "\overline{x}_{\tilde{L}}"] & \spec \OO_K \arrow[d, "q"]
\\ 
\spec K \arrow[r] &\mathcal{C} \arrow[r, "\overline{x}"] \arrow[d, "\pi"] & BG_{\spec \OO_K}\\
&\spec \OO_K&
\end{tikzcd}
    \label{fig:tuning-module-completed}
\end{equation}
where $p,q$ are quotient maps. In particular, $q$ is the quotient by the trivial action of $G$.
Note that \(V = q^* \mathcal{V}\) and for every place \(v\) of \(K\), we get a local version of this diagram.

By \cite[Definition 2.21]{ESZB}, the discussion following it and by \cite[Proposition B.10]{ESZB}, the height of \(x\) with respect to \(\mathcal{V}\) satisfies
    \[
    \height_{\mathcal{V}}(x) = -\deg\left(\pi_* \ol{x}^*(\mathcal{V}^{\vee}) \right) = - \deg\left(\ol{x}^*\mathcal{V}^{\vee}\right) + \sum_{\substack{v \in M_{K}\\ \text{finite}}} \locdiscrep,
    \]
    where the local discrepancies \(\locdiscrep\), 
    which essentially measure the difference between the vector bundle on the tuning stack and its pushforward to \(\spec \OO_K\),
    can be computed as 
    \[
    \locdiscrep = \frac{1}{\#G} \log \left| \frac{p^*\ol{x}^* \mathcal{V}^{\vee}_v  }{ \pi_* \left(\ol{x}^{*} \mathcal{V}_v^{\vee}\right) \otimes_{\OO_{K_v}} \OO_{\tilde{L}_v} } \right|.
    \]
    Here \(\tilde{L}_{v}\) denotes the \'etale algebra \(\prod_{w|v} \tilde{L}_w\). 
     Further, since \(V = q^*\mathcal{V}\) is the trivial vector bundle, \(\deg(\ol{x}^*\mathcal{V}^{\vee}) = 0\). Therefore, the height decomposes into a product of local factors, and we may work locally. For notational convenience, we drop the localizing subscripts from now on: we replace \(K_v, L_v\) and \(\tilde{L}_v\) by \(K, L \) and \(\tilde{L}\) in what follows.

Write \(L = K(\alpha)\) and let \(\alpha_1 := \alpha, \alpha_2, \ldots, \alpha_n\) be  the Galois conjugates of \(\alpha\). Note that by the way \(L\) is defined, the permutation action of \(G\) on the \(\alpha_i\)'s is determined by the embedding \(\rho\), and vice-versa. Denote by \(\sigma_i \in G\) an element that takes \(\alpha_1 \mapsto \alpha_i\). If \(\sigma_i\) and \(\sigma_i'\) are two choices sending \(\alpha_1 \mapsto \alpha_i\), then \(\sigma_i^{-1}\sigma_i'\) stabilizes \(\alpha_1\) and therefore belongs to \(\Stab_{\rho(G)}(1)\). In particular, if \(\beta \in L\), then \(\sigma_i(\beta) = \sigma_i'(\beta)\).

Since each \(\sigma_i\) fixes \(K\), 
\[
V = \OO_K^{\oplus n} \cong \oplus_{i=1}^n \sigma_i (\OO_K) 
\]
as \(\OO_K\) modules. Similarly, by the commutativity of the diagram \eqref{fig:tuning-module-completed}, we have the following isomorphisms of  \(\OO_{\tilde{L}}\) modules.
\[
p^*\overline{x}^*\mathcal{V} = \overline{x}_{\tilde{L}}^{*}V \cong V \otimes_{\OO_K} \OO_{\tilde{L}} \cong \OO_{\tilde{L}}^{\oplus n} \cong \oplus_{i=1}^n \sigma_i(\OO_{\tilde{L}}),
\]
where the  \(\OO_{\tilde{L}}\)-module structure on $\oplus_{i=1}^n \sigma_i(\OO_{\tilde{L}})$ is given by  \(x \cdot (y_1, y_2, \ldots y_n) = (\sigma_1(x)y_1, \sigma_2(x)y_2, \ldots \sigma_n(x)y_n )\).

Let \(g \in G\). If \(g(\alpha_i) = \alpha_j\) (that is, \(\rho(g)(i) = j\)), then \(g^{-1} \sigma_j(\alpha_1) = g^{-1}(\alpha_j) = \alpha_i = \sigma_i(\alpha_1)\).  Therefore, \(g^{-1}\sigma_j\) and \(\sigma_i\) belong to the same coset of \(\Stab_{\rho(G)}(1)\). Thus we have a \(G\) action on \(p^* \ol{x}^* \mathcal{V} = \mathcal{V} \otimes \OO_{\tilde{L}}\) given by
\[
g \cdot (\sigma_1(x_1), \sigma_2(x_2) \ldots , \sigma_n(x_n)) = \left(g^{-1}\sigma_{\rho(g)(1)}(x_1), \ldots, g^{-1}\sigma_{\rho(g)(n)}(x_n) \right),
\]
whose invariants are given by tuples \( (\sigma_1(x_1), \sigma_2(x_2) \ldots , \sigma_n(x_n)) \) such that \(\sigma_i(x_i) = g^{-1}\sigma_{\rho(g)(i)}(x_i)\) for each \(i\), that is, \(x_i \in L\). In other words, the \(G\)-invariant submodule of \(p^* \ol{x}^* \mathcal{V}\) is precisely the  {tuning module} as described in the proof of \cite[Theorem 4.8]{yasuda-wood}.

Now, since \(p^*\ol{x}^*\mathcal{V}\) is the pullback of a vector bundle along the quotient map \(p\), it is \(G\)-equivariant by definition. Thus, by descent theory (see for instance \cite[\S 4.4]{vistoli-fag}):
\[
\left(p_* p^* \ol{x}^* \mathcal{V}\right)^G \cong \ol{x}^*\mathcal{V}. 
\]
Taking \(\pi_*\) on both sides tells us that \(\pi_*\ol{x}^*\mathcal{V}\) is the tuning module considered as an \(\OO_K\) module. We may now apply the proof of \cite[Theorem 4.8]{yasuda-wood} to conclude that \(\locdiscrep = \frac{1}{2} v(\Delta(L/K))\log(q_v)\), where \(q_v\) is the size of the residue field at \(v\).
\end{proof}

Next, we  consider the case of \(B\mu_n\) and recall the approach of
\cite[\S 3.2]{ESZB} for computing heights on \(B\mu_n\). While the height in the case of the constant group scheme \(\ZZ/n\ZZ\) falls into the purview of the previous section, i.e., the height is exactly the discriminant, the height in the non-constant case may differ from the discriminant at finitely many places. In particular, this height is a  {quasi-discriminant} in the sense of \cite{dardathesis}. We will explore quasi-discriminants in a later section. 

\begin{theorem}
\label{thm:ESZB-height-Bmu-n}
    Let \(\chi: \mu_n \to \GG_m \) be a primitive character. Let \( \mathcal{V} = \left( \OO_K^{\oplus n} , \oplus_{a=0}^{n-1} \chi^a\right)\) be the corresponding rank \(n\) vector bundle {on $B\mu_{n,\OO_K}$}. Let \(x \in B\mu_n(K)\) correspond to an \'etale algebra \(L\).
        Assume that \(L/K\) does not contain any intermediate cyclotomic extensions, i.e., \(L\) and \(K(\zeta_n)\) are linearly disjoint over \(K\). Then, there is a constant \(C\) depending only on \(n\) and \(K\), and \(c_x\in\mathbb R\) satisfying \(|c_x| < C\) such that
\[
\height_{\mathcal{V}}(x) = c_x + \frac{1}{2}\log|\Delta(L/K)|. 
\]
\end{theorem}

\begin{proof}
    We claim that over \(K(\zeta_n)\), the representation
\begin{equation}
\label{eq:representation mu_n}
\oplus_{a=0}^{n-1} \chi^{a} : \mu_n \to \oplus_{a=0}^{n-1} \GG_m \hookrightarrow \GL_n(\OO_K)
\end{equation}
is isomorphic to a permutation representation of \(\ZZ/n\ZZ\) acting on itself. Indeed, if \(\chi\) sends \(\zeta_n \mapsto \zeta_n^k\) for some positive integer \(k\) with \(\gcd(k,n)=1\), then the composition \eqref{eq:representation mu_n}
sends 
\[
\zeta_n \mapsto (\zeta_n^0, \zeta_n^k, \zeta_n^{2k}  \ldots ,\zeta_n^{(n-1)k} ) \mapsto [e_1, e_k, e_{2k}, \ldots, e_{(n-1)k} ],
\]
where \(e_i\) is the standard \(i\)-th basis vector, where we take \(i\) modulo \(n\). 
This is a permutation matrix of order \(n\). By Lemma \ref{rmk:mu_n isom Z/nZ}, the group schemes \(\underline{\ZZ/n\ZZ}\) and \(\mu_n\) are isomorphic over \(K(\zeta_n)\). Further, the action described above is the same as \(\ZZ/n\ZZ\) acting on itself by 
{$a\cdot b=ka+b$.}
In particular, over \(K(\zeta_n)\) this becomes a permutation representation associated to a constant group scheme, thus falling into the setting of Proposition \ref{proposition:ESZB-discriminant}. Let \(y\) denote the base change of $x$ to $K(\zeta_n)$ and let $L(\zeta_n)$ denote the corresponding \'etale algebra.

By Proposition \ref{proposition:ESZB-discriminant},
\[
\height_{\mathcal{V}_{K(\zeta_n)}}(y) = \frac{1}{2}\log|\Delta(L(\zeta_n)/K(\zeta_n))|.
\]
Let \(\mathcal{S}\) denote the set of primes ramified in \(K(\zeta_n)/K\). By \cite[Corollary III.2.12]{neukirch}, it is precisely the set of
primes dividing \(n\). By the theory of local discrepancies in \cite[\S 2]{ESZB} and the fact that \(\mathcal{V}\) is the trivial vector bundle on \(\spec \OO_K\), one can decompose the height on \((B\mu_n)_K\) as \[\height_{\mathcal{V}}(x) = \sum_{v \in \mathcal{S}} \delta_{\mathcal{V}, v}(x) + \sum_{v \notin \mathcal{S}} \delta_{\mathcal{V}; v}(x). \]
Let \(\mathcal{V}_n\) denote the pullback of \(\mathcal V\) 
to \(B\mu_{n,K(\zeta_n)}\). Then by \cite[Proposition 2.24]{ESZB}, over \(K(\zeta_n)\)
\begin{align*}
    \height_{\mathcal{V}_n}(y) &=  \sum_{w|v, v \in \mathcal{S}} \delta_{\mathcal{V}_n; w}(y) + [K(\zeta_n):K]\sum_{v \notin \mathcal{S}} \delta_{\mathcal{V}; v}(x)\\
    &= \sum_{w|v, v \in \mathcal{S}} \delta_{\mathcal{V}_n; w}(y) + [K(\zeta_n):K] \left(\height_{\mathcal{V}}(x) - \sum_{v \in \mathcal{S}} \delta_{\mathcal{V}, v}(x) \right).
\end{align*}
Since \(\mathcal S\) is a fixed finite set depending only on \(n\) and \(K\), by \cite[Proposition 2.25]{ESZB} there is a constant \(C\) depending only on \(n\) and \(K\), such that
\[
\height_{\mathcal{V}}(x) + c_x= \frac{1}{[K(\zeta_n):K]}\height_{\mathcal{V}_n}(y) = \frac{1}{2[K(\zeta_n):K]}\log|\Delta(L(\zeta_n)/K(\zeta_n))|,
\]
and \(|c_x| < C.\) Further, by properties of discriminants in towers and compositums \cite[\S III.2]{neukirch},
\[\Nm_{K(\zeta_n)/K}(\Delta(L(\zeta_n)/K(\zeta_n))) = \Delta(L/K)^{[K(\zeta_n):K]},\]
{where $\Nm_{K(\zeta_n)/K}$ denotes the norm.}
We see that \(\height_{\mathcal{V}}(x)\) is equal to \(\frac{1}{2} \log |\Delta(L/K)|\) upto a constant that is bounded on \(B\mu_n(K)\).
\end{proof}

\begin{remark}
In \cite[\S 3.2]{ESZB}, the authors describe the height on \(B\mu_{n,\QQ}\) with respect to the line bundle \(\chi\) itself (as opposed to the vector bundle used in Theorem \ref{thm:ESZB-height-Bmu-n} above). Suppose for simplicity that \(\chi\) comes from the standard representation, i.e. \(\chi(\zeta_n) = \zeta_n \). By considering
\(B\mu_n\) as a zero-dimensional weighted projective stack,\footnote{See \cite[Chapter 3]{dardathesis} for a detailed definition of weighted projective stacks.}
the authors show that if \(a \in \QQ^*/(\QQ^*)^n\) is an \(n\)-th power free integer, then the height of the corresponding \(\mu_n\)-torsor over \(\QQ\) is \(\log|a|^{1/n}\) up to a bounded function.
While this is not directly relevant to this article, it gives a nice description of that height in terms of the torsor, and moreover shows ways in which one can use sums and powers of \(\chi\) to construct
various height functions 
on \(B\mu_n\), 
i.e., number field counting invariants.
\end{remark}

\subsubsection{Darda's height}
\label{subsec:darda-height}

In this subsection, we briefly describe the height defined in \cite{dardathesis}. In his work, Darda defines the notion of a  {quasi-toric} height on a weighted projective stack. 
He then specializes the height to the case of \(B\mu_n\), viewing it as the \(0\)-dimensional weighted projective stack $\PP(n)$.

As before, let $K$ be a number field and let $v$ denote a place of $K$. Let $r$ denote the smallest prime factor of $n$.
\begin{definition}[{\cite[Lemma 9.1.3.1]{dardathesis}}]
\label{defn:darda-local-factors}
    For every $n\in\mathbb N$ and every place $v$ of $K$, we define a
function $\fdiscn_{v}: K_{v}^{\times} \to \R_{\ge 0} $ as follows:
\[
\fdiscn_{v}(a) = \begin{cases}
    |a|_v^{1/n} \left|\left(\Delta((K_{v}[t]/(t^n - a)) / K_{v})\right)^{1/(n^2-n^2/r)} \right| & \text{ if $v$ finite,}\\
    |a|_v^{1/n} & \text{ otherwise,}
\end{cases}
\]
where $|\cdot|_v=q_v^{-v(\cdot)}$ denotes the $v$-adic absolute value if $v$ is finite, the real absolute value if $v$ is real, and the square of the complex absolute value if $v$ is complex, and \(\mid \cdot \mid\) denotes the absolute ideal norm. 
\end{definition}

\begin{lemma}
\label{lemma:darda-discriminant}
    For $x \in B\mu_n(K)$, define the discriminant height as
    \[
    H^{\Delta}(x) := \prod_{v \in M_K} \fdiscn_{v}(x).
    \]
    Then $H^{\Delta}(x)^{n^2 - n^2/r}$ is precisely the discriminant of the torsor \(K_a\) corresponding to $x$.
\end{lemma}

\begin{proof}
    Let \(a \in K^{\times}\) be an element such that \(K_a\) represents the class of \(x \in B\mu_n(K) \cong K^{\times}/(K^{\times})^n\).
    By the product formula, the multiplicativity of the ideal norm and the local decomposition of the discriminant, we have that $H^{\Delta}(x)$ is $| \left(\Delta((K[t]/(t^n - a)) / K)\right)^{1/(n^2-n^2/r)}|$.
\end{proof}

In his PhD thesis \cite{dardathesis}, 
{Darda defined quasi-discriminant heights as functions that agree with the discriminant height at all but finitely many places\footnote{More precisely, a quasi-discriminant height coming from a degree \(k\) family is a product of local functions \(f_v\) that agree with $\left(\fdiscn\right)^k$ at all but finitely many places. }. He shows that any quasi-discriminant height has the Northcott property, and gives a general framework to use height zeta functions for counting points of bounded quasi-discriminant height on weighted projective stacks.}

\subsubsection{The Darda--Yasuda perspective}
\label{subsec:darda-yasuda-heights}

In their work \cite{darda-yasuda22} on a stacky version of the Batyrev--Manin conjecture, 
Darda and Yasuda define the height of a point on a stack \(\mathcal{X}\) with respect to a ``raised line bundle'', which is the data of a line bundle on the coarse space of \(X\), along with a ``raising function'', that captures the stackiness of \(\mathcal{X}\). In this subsection, we briefly outline their definition of height on \(B\mu_n\) as in \cite{darda-yasuda22} and \cite{darda-yasuda-torsors1}, and describe its relation to the discriminant.

  \begin{definition}[{\cite[Definition 2.3]{darda-yasuda22}}]
        Let \(\mathcal{X}\) be a tame Deligne-Mumford stack.\footnote{See \cite[Proposition 2.1]{bresciani-vistoli} for the definition of a tame stack.} Define the stack of 0-jets as
        \[
        \JetsX = \bigsqcup_{l>0} \underline{\Hom}_K^{rep}\left( B{\mu_l}_K, \mathcal{X}\right),
        \]
        where \(\underline{\Hom}_K^{rep}\) denotes the stack of representable \(K\)-morphisms. We refer the reader to \cite[Definition 5.1.5]{Olsson-Book} for a definition of representability. However for this article, one may think of this as requiring that the induced maps on automorphism groups of the points be injections (see for instance \cite[Lemma 3.2.2 (b)]{Cesnavicius}).
       Henceforth, we will only keep track of the connected components of this stack, \(\pi_0(\JetsX)\). Following \cite{darda-yasuda22} we denote by \({\pi_0}^*(\JetsX)\) the set of {twisted sectors}, i.e., elements of \(\pi_0(\JetsX)\) coming from \(l>1\).
    \end{definition}
    
When \(\mathcal{X} = B\mu_n\), the disjoint union ranges over \(l|n\), and for each such \(l\), \({\pi_0}(\underline{\Hom}^{rep}_K (B\mu_{l, K}, B\mu_{n,K}))\) can be identified with the group of group homomorphisms \(\mu_l \to \mu_n\) that are isomorphisms onto their image. 
This group is isomorphic to \((\ZZ/l\ZZ)^{\times}\). Thus \(\pi_0(\Jets B\mu_n) \) can be identified with \(\ZZ/n\ZZ\). For more on this, see \cite[Example 2.15]{darda-yasuda22}. 

For \(B\mu_n\) over a number field \(K\), one may think of the set of twisted sectors as keeping track of the possible (tame) inertia groups of \(\mu_n\)-torsors over \(K\). In fact, this is true more generally for \(BG\), where \(G\) is a finite \'etale group scheme.

   \begin{definition}
   \label{defn:raising-function}
        A raising datum, which by abuse of notation we will call a  {raising function}, on \(BG(K)\) is a collection \((c_v)_{v \in M_K}\) of continuous functions 
        \[
        c_v \colon BG(K_v) \to \RR_{\ge 0}
        \]
        such that for all but finitely many \(v \in M_K\), the function \(c_v\) factors through  \( BG(K_v) \xrightarrow{res_v} {\pi_0}(\Jets BG ) \xrightarrow{c} \RR\), for some function \(c\) satisfying \(c(BG) = 0\). Here \(res_v\) is a residue map that takes a point \(x_v \in BG(K_v)\) to the induced map \(B\mu_l \to BG\) as in the discussion following \cite[Lemma 2.16]{darda-yasuda22}. 
          \end{definition}
Each residue map essentially assigns to a local \(G\) extension its (tame) inertia group. In particular, the maps \(res_{v}\) are only defined outside of a finite set of places \(v \in M_K\). While the definition of a raising function might appear nebulous at first glance, it performs a very natural function:  it assigns local values at \(v\) to an \'etale algebra, depending on the tame inertia groups above the place \(v\). 

Definition \ref{defn:raising-function} is related to {counting functions} in \cite{wood-probabilities} and class functions in \cite{EllenbergVenkatesh2005}.  
 
\begin{definition}
    Maintaining the notation above, the height corresponding to a raising function $c=(c_v)_{v \in M_K}$ is defined as \[\Height_c(x) = \prod_{v \in M_K}q_v^{c_v(x_v)},\]
    where \(q_v\) is the size of the residue field at \(v\) if \(v\) is finite, and 
    {Euler's number}
    \(e\) if it is infinite.
\end{definition}
In particular, on \(BG\) for a constant group scheme \(G\), {by taking the local factor \(c_v\) to be $v$-adic valuation of the local discriminant at the place \(v\), 
one obtains a raising function}
\(c\) for which \(\Height_c(x)\) is the discriminant of the corresponding \'etale algebra. 
{See for instance \cite[\S 2.5.3]{darda-yasuda-torsors1}}.

\begin{remark}
    For a more general stack \(\mathcal{X}\), the height defined in \cite{darda-yasuda22} has a component coming from the coarse space of \(\mathcal{X}\). However, in the case of \(BG_{K}\), the coarse space is just a point, so we do not see this contribution.
\end{remark}

\begin{remark}\label{DYframe}
One of the advantages of the Darda--Yasuda framework is that it allows for height functions that do not necessarily arise from a vector bundle. For instance, \cite[Example 4.11]{darda-yasuda22} defines a raising function on \(B\ZZ/p\ZZ\) whose associated height cannot agree with the height with respect to  {any} vector bundle in the sense of \cite{ESZB}. On the other hand, the Darda--Yasuda framework currently only works for tame stacks, while the Ellenberg--Satriano--Zureick-Brown framework works for wild stacks as well. For general stacks, the work in \cite{landesman2023stacky} illustrates that there are heights studied widely that do not fit into either framework: it is shown that the Faltings height on the moduli stack of elliptic curves in characteristic three does not come from a vector bundle in the sense of \cite{ESZB}, but this stack is 
not tame
and so does not fit into the Darda--Yasuda framework.
\end{remark}

\section{Malle's conjecture}
\label{sec:malles-conjecture}
In this section, we recall Malle's conjecture in \S \ref{subsec:classical-malle} and we compare it with the recent stacky conjectures of \cite{ESZB} in \S \ref{subsec:eszb-conjecture} and \cite{darda-yasuda22} in \S \ref{subsec:darda-yasuda-conjecture}.

\subsection{Classical Malle's conjecture} 
\label{subsec:classical-malle}

In \cite{malle1} and \cite{malle2}, Malle formulated a conjecture on the order of growth of the number of number fields of degree \(n\) over a given base number field $K$ with given Galois group $G \hookrightarrow S_n$ ordered by discriminant. According to this conjecture
the order of growth is expected to depend on two constants associated to the group $G$, denoted by $a(G)$ and $b(K,G)$ and defined below. 

We view $G$ as a subgroup of the permutation group $S_n$ such that $G$ acts transitively on $\{1, \ldots, n\}$. For $g \in G$, let 
\[
n(g) := \textup{number of orbits under the action of } g \textup{ on } \{1, \ldots, n\}.
\]
Note that $n(g)$ depends only on the conjugacy class of $g$. The number $n-n(g)$ is called the \textit{index} of $g$. Define
\[
a(G)=(\min\{n-n(g):g\in G, g\neq 1\})^{-1}.
\]
Note that \(a(G)\) depends on the embedding \(G \hookrightarrow S_n\), and is equal to 1 if and only if the image of \(G\) contains a transposition. 

The constant $b(K,G)$ may depend also on the base number field $K$. Let $\Gamma = \gal(\Qbar/\QQ)$ denote the absolute Galois group of $\QQ$. Let $\chi: \Gamma \to \hat{\ZZ}^{\times}$ denote the cyclotomic character, i.e., for every \(\sigma \in \Gamma\) and $m\in\NN$, define \(\chi(\sigma)\) by \(\sigma \cdot\zeta_m = \zeta_m^{\chi(\sigma) \text{ mod }m}\) for any primitive \(m\)-th root of unity \(\zeta_m\). Since for any \(m\), \(\chi(\sigma)\) is invertible modulo \(m\), the image of $\chi$ does indeed land inside \(\hat{\ZZ}^{\times}\). Define the action of $\Gamma$ on the set $\Cl(G)$ of conjugacy classes of $G$ as follows.
For $\sigma \in \Gamma$ and $g \in G$, set $\sigma \cdot g := g^{\chi(\sigma)}$. This gives an action on \(\Cl(G)\) and further, \(g\) and \(g^{\chi(\sigma)}\) have the same index. 
The subgroup $\Gamma_K = \gal(\Qbar/K) \leq \Gamma$ thus also acts on \(\Cl(G)\), and we can define
\[
b(K,G)=\#\left(\{[g] \in \Cl(G): n-n(g) = a(G)^{-1}\}/\Gamma_K\right),
\]
the number of $\Gamma_K$-orbits of conjugacy classes of $G$ with minimal index. In \cite[Lemma 2.2]{malle2}, Malle showed that if \(a(G)\) is 1, then \(b(K,G)\) is also 1 for all number fields $K$. The dependence of the constant \(b(K,G)\) on the field \(K\) highlights the effect of having roots of unity in the base field \(K\). If \(\zeta_m \in K\), then for \( \sigma \in \Gamma_K\), \(\chi(\sigma) \equiv 1 \mod m\). So if \(m|\#G\), then \(\Gamma_K\) fixes more elements and thus \(b(K,G)\) is smaller. For a more precise explanation, see \cite[Example 2.1]{malle2} or Example \ref{example:malle-constants-zmodn} below.

\begin{conjecture}[Weak Malle's conjecture {\cite{malle1}}]
Let $K$ be a number field with algebraic closure $\overline K$. Let $G$ be a subgroup of the  group $S_n$ of permutations of $n$ elements such that $G$ acts transitively on $\{1,\dots,n\}$. Let $M(K,G;B)$ be the number of finite field extensions $L/K$ of degree $n$ contained in $\overline K$ such that the Galois closure of $L$ has Galois group $G$ and the norm of the discriminant of $L/K$ is bounded above by $B>0$. Then
    for all \(\varepsilon>0\), there exist positive real constants \(c_1(K,G)\) and \(c_2(K,G, \varepsilon)\) such that 
\begin{equation}
    \label{eq:weak-malle}
    c_1(K,G) B^{a(G)} \le M(K,G;B) \le B^{a(G) + \varepsilon}.
\end{equation}
\end{conjecture}

A stronger version of the conjecture, stated in \cite{malle2}, asserts the following:
\begin{conjecture}[Strong Malle's conjecture \cite{malle2}]
\label{conjecture:strong-malle}
Under the same assumptions,
\begin{equation}
\label{eq:strong-malle}
    M(K,G;B) \sim c(K,G) B^{a(G)}(\log B)^{b(K,G)-1},
\end{equation}
for a positive constant $c(K,G)$.
\end{conjecture}

To our knowledge there is no general conjecture for the value of the leading constant \(c(K, G)\). A conjectural formula for \(c(K, S_n)\) has been given by Bhargava in \cite{bhargava-mass07}, which is compatible with the known values for $n\leq 5$. In the cases where $G$ is abelian and $G = \ZZ/n\ZZ$, the exact values of \(c(K, G)\) are computed  in \cite{wright} and \cite{dardathesis}, respectively. 

Several variations of Malle's conjecture have been studied in the last few decades. For instance, one may generalize Malle's conjecture to finite extensions of \(\FF_q(t)\) in the case when \(q\) is coprime to \(|G|\). In this case, the discriminant is replaced by \(q\) raised to the degree of the ramification divisor. See \cite{EllenbergVenkatesh2005} for a description of this problem in terms of counting \(\FF_q\) points on Hurwitz spaces, as well as a heuristic argument for why it must hold.
 One can formulate Malle-type conjectures for invariants other than the discriminant, such as the conductor or any  {counting function} (\cite{wood-probabilities}). The work in this area is vast and growing, but for the purpose of this article we will restrict our attention to discriminants.
In the counting function \(M(K,G;B)\), one may instead count  {isomorphism classes} of extensions. This would change the leading constant of the asymptotic formula by a factor \(\#\aut(G)\).
In \cite{malle2}, Malle formulates the weak conjecture also for number fields that are unramified outside a finite set \(S\). The \(a\)-invariant in this case is the same as before, and is independent of \(S\).

Before proceeding to what is known about this conjecture, we give some examples of what it says in some cases of interest in this article. 

\begin{example}[Regular representation]
    \label{example:malle-regular-rep}
    If $n=\#G$ and the inclusion $G\subseteq S_{n}$ is induced by the action of $G$ on itself by the group operation, then $M(K,G;x)$ is the number of degree $n$ Galois extensions of $K$ of Galois group $G$ and discriminant bounded by $x$. In this case, $a(G)=\frac{1}{n-n/r}$ where $r$ is the smallest prime number dividing $n$, i.e., the smallest order of a nonzero element of $G$.
\end{example}

\begin{example}[Cyclic groups]
\label{example:malle-constants-zmodn}
    If $G=\ZZ/n\ZZ$ with the inclusion in $S_n$ given by the group operation as in 
    \eqref{eq:subgroup}, then Malle's conjecture concerns Galois field extensions of degree $n$, i.e., connected $K$-points of $B(\underline{\ZZ/n\ZZ})$. As in Example \ref{example:malle-regular-rep}, $a(G)=\frac 1{n-n/r}$, where $r$ is the smallest prime number dividing $n$. To compute \(b(K,G)\), note that \(\ZZ/n\ZZ\) contains \(\phi(r)=r-1\) elements of order \(r\), all of which have the same index. If \(K\) contains all the \(r\)-th roots of unity, then all of these elements (viewed as conjugacy classes) are fixed by \(\Gamma_K\) and so \(b(K,G)=1\). If \(K=\QQ\), then \(b(K,G) = r-1\). A similar argument shows that in general, \(b(K,G) = (r-1)/[K(\zeta_r):K].\)
\end{example}

\begin{remark}
\label{rmk:counting-torsors-versus-malle}
Conjectures \ref{eq:weak-malle} and \ref{conjecture:strong-malle}  concern only \(K\)-rational points on \(BG_K\) that correspond to  {connected} torsors, where \(G\) is a  {constant} group scheme over \(\spec K\). One may instead set up a different counting function
    \begin{equation*}
        T(K,G;B) := \#\{x \in BG(K) \mid |\Delta(x)| < B \},
    \end{equation*}
    which for constant \(G\) would count the number of isomorphism classes of \'etale \(G\)-algebras of bounded absolute discriminant.  While the space of irreducible polynomials is Zariski dense in the space of all polynomials, it is generally not true that \(M(K,G;B)\) makes up \(100\%\) of \(T(K,G;B)\). This only happens in the case when \(G\) has prime exponent (see \cite[\S 2.1]{wood-probabilities}).\footnote{In the terminology of \cite{wood-probabilities}, the discriminant is a ``fair'' counting function only when \(G\) has prime exponent.}
\end{remark}

\subsubsection{Known results and counter-results}

The strong version of Malle's conjecture is known to be true in some cases and false in some others. Some of the earliest work in this area dates back to work of Davenport--Heilbronn \cite{MR491593} for \(G= S_3\), and M\"{a}ki \cite{Maki-thesis} for abelian \(G\) with the regular representation and \(K=\QQ\). In \cite{wright}, Wright proved Malle's conjecture for (Galois) abelian extensions over arbitrary number fields and for function fields with characteristic coprime to \(|G|\). In recent work, Darda \cite{dardathesis} proved the following result.
 \begin{theorem}[{\cite[Corollaries 9.2.5.9, 9.2.5.10]{dardathesis}}]
            Let $H$ be a quasi discriminant height coming from a degree $n$ family, i.e., one whose local factors differ from \((H^{\Delta})^n\) (Lemma \ref{lemma:darda-discriminant}) at finitely many places. 
            Let $r$ be the smallest prime dividing $n$. Then, the number of $K$ points of $B\mu_n$ with height bounded by $B$ grows asymptotically like
            \[
            C_{K,H,r} B (\log B)^{r-2},
         \]
            where $C_{K,H,r}>0$ is a positive constant. In particular, in light of Lemma \ref{lemma:darda-discriminant},
            \[
            T(K, \mu_n; B) \sim C_{K,H^{\Delta},r} B^{1/(n-n/r)} (\log B)^{r-2}.
             \]
        \end{theorem}
When the base field contains all the \(n\)-th roots of unity, \(\mu_n\) and \(\ZZ/n\ZZ\) are isomorphic as group schemes and the powers of \(B\) and \(\log B\) in \cite{wright} and \cite{dardathesis} agree. In light of Remark \ref{rmk:counting-torsors-versus-malle}, note that the constants in front in general, do not.

Malle's conjecture has been proved in several non-abelian cases, including work of Bhargava \cite{bhargava-s4, bhargava-s5} and Bhargava--Shankar--Wang \cite{bsw2015geometryofnumbers} for \(S_4\) and \(S_5\), Wang \cite{wang2021} for certain compositums of fields, Cohen--Diaz-y-Diaz--Olivier \cite{MR1918290} for \(D_4\) extensions, among many others. Since it is impossible to do justice to the vast body of literature in this area within the scope of this article, we refer the reader to \cite{wang2021}, \cite{alberts21}, and \cite{wood-probabilities} for a comprehensive survey. The weak form of Malle's conjecture was proven by Kl\"{u}ners and Malle  \cite{kluners-malle} for nilpotent groups with the regular representation; the strong upper bound was proved by Kl\"{u}ners \cite{kluners-upper}, and a recent result of Koymans and Pagano \cite{MR4555793} proves the strong form of Malle's conjecture for a large class of nilpotent groups. In \cite{Alberts20}, Alberts proved a general upper bound in the case of solvable groups.

\begin{remark}
    Let \(K_a/K\) denote the \'etale algebra described in \S 2. Then \(\gal(K_a/K)\) is contained in some extension \(E\) of \(\ZZ/n\ZZ\) by \((\ZZ/n\ZZ)^{\times}\) and is in particular solvable. 
    There is a finite-to-finite relation between \'etale algebras of the form \(K_a\) and \'etale algebras with Galois group contained in \(E\). Further, there are finitely many possibilities for \(E\).
    Thus one may at the very least derive weak bounds on counts of \'etale algebras of the form \(K_a\) from the known results about solvable groups.
\end{remark}

\begin{remark}
If $K$ is a number field, we can assume that $a$ is an algebraic integer.    By Eisenstein's criterion, if the polynomial $t^n-a$ is reducible, then $a$ is a squareful number, that is, every prime appearing in the factorization of the ideal $(a)$ into prime ideals has exponent at least 2. The sharp criterion for irreducibility can be found in \cite[Theorem VI.9.1]{Lang-Algebra}. If $K=\QQ$, identifying $\QQ^\times/(\QQ^\times)^n$ with the set of $n$-th power free integers (positive if $n$ is odd),
the set of irreducible polynomials $t^n-a$ with constant term of size at most $B$
has growth $\gg B$, while the set of reducible polynomials $t^n-a$ with constant term of size at most $B$
grows at most as $O(B^{\frac 12})$. Hence, according to this measure, most $\QQ$-points of $B\mu_n$ correspond to \'etale algebras $K_a$ that are fields. 
\end{remark}

Of course, the elephant in the room for any discussion on Malle's conjecture is the fact that it is known to be false in some cases. In \cite{klunerscountereg}, Kl{\"u}ners showed that for \(G = C_3 \wr C_2 \hookrightarrow S_6\), the strong form of Malle's conjecture is incorrect. In this case, the conjecture predicts that \(M(\QQ, G; B) \asymp B^{1/2}\). However, Kl\"uners shows that one actually obtains \(B^{1/2} \log B\). In particular, only the weak form of the conjecture still holds.
He also points out that if one instead counts \(C_3 \wr C_2\) fields that do not contain intermediate cyclotomic extensions, one does indeed obtain the asymptotic predicted in Conjecture \ref{conjecture:strong-malle}.
Following this, T\"urkelli \cite{Turkelli} suggested a modified version of the conjecture over function fields \(\FF_q(t)\) in the case \(\gcd(q, |G|) = 1\). Containing intermediate cyclotomic extensions in the number field case is analogous 
to having non-trivial constant subextension in the function field case.
T\"urkelli's idea was to keep track of the constant subextensions of the extensions of function fields and modify the \(b\)-constant accordingly. See \cite[\SS 2, 5]{Turkelli}.
The counterexamples to Malle's conjecture are believed to stem from the existence of ``accumulating'' sets (akin to breaking thin sets in \cite{darda-yasuda22}): these are ``small'' sets that contribute more than expected to the counting function.

\subsubsection{Malle-type conjectures for torsors}
\label{subsec:classical-malle-torsors}
One useful perspective on counting number fields with a given Galois group (and indeed this is a perspective used in several papers mentioned above), is that this amounts to counting the number of continuous surjections $\Gal{K}\to G$.
Similiarly, counting \'etale algebras  amount to counting continuous homormorphisms. Note that if $\Gal{K}$
acts on \(G\) trivially, then $H^1(\Gal{K},G)=\Hom(\Gal{K},G)$.. 
This leads to the question of counting classes in the Galois cohomology group, $H^1(\Gal{K},G)$ 
for more general group actions and is related to counting points on \(BG\) for non-constant group schemes. In this context, there have been several attempts at formulating a ``twisted'' version of Malle's conjecture for non-constant group schemes \(G\). Some classical approaches and results can be found in \cite{alberts21} or \cite{alberts-odorney}, while a more stacky approach is taken in \cite{darda-yasuda-torsors1} or \cite{darda-yasuda-torsors2}.

\subsection{The stacky Batyrev--Manin--Malle conjecture}
\label{subsec:stacky-batman-malle}
We end our article with a discussion on Malle-type conjectures for torsors, i.e., for points on the stack \(BG\). When \(G\) is a constant group scheme, these conjectures specialize to the classical versions of Malle's conjecture described above in \S \ref{subsec:classical-malle}. When \(G\) is non-constant, the story is more complicated:
in \cite{darda-yasuda-torsors1}, the authors give such a conjecture, although in this case the height of a point does not admit a straightforward description as the discriminant of an \'etale algebra. 

In this section, we focus our attention on the constant group scheme \(\underline{\ZZ/n\ZZ}\). We briefly describe work of \cite{ESZB} and \cite{darda-yasuda22} in the context of Malle's conjecture, recalling the definitions of the invariants that they define and showing that they agree with the invariants for Malle's conjecture for \(B(\ZZ/n\ZZ)\). We remark that their conjectures apply to more general stacks, but that is beyond the scope of this article.

In what follows, we assume that the base field \(K\) contains a primitive \(n\)-th root of unity, i.e., that \(B\mu_n\) and \(B(\ZZ/n\ZZ)\) are isomorphic over \(K\). For the conjectures for general finite \(G\) see \cite{darda-yasuda-torsors1}. 

\subsubsection{The vector bundle perspective}
\label{subsec:eszb-conjecture}

For a ``nice'' stack \(\mathcal{X}\), the authors in \cite{ESZB} define a notion of being ``generically`' bounded below, that helps them establish finiteness properties.

\begin{definition}
\label{defn:gen-bounded-below}
    A function \(f: \mathcal{X}(\ol{K}) \to \RR\) is said to be  {generically bounded below} if there is some proper substack \(Z \hookrightarrow \mathcal{X}\) and a constant \(B\) such that \(\{ x \in \mathcal{X}(\ol{K}) \mid f(x) < [K(x):K] \cdot B \} \subset Z(\ol{K})\), where $K(x)$ is the field of definition of $x$.
\end{definition}

When \(\mathcal{X} = BG_{\OO_K}\) for a smooth finite constant group scheme \(G\) over a number field \(K\), \(\mathcal{X}\) has no proper closed substacks, and this notion reduces to a more common notion of being bounded below. 

\begin{lemma}
     A function $f:BG(\overline K)\to\RR$ is generically bounded below as in Definition \ref{defn:gen-bounded-below} if and only if there is a constant $B$ such that for all \(x \in BG(\overline{K})\), $f(x)\geq [K(x):K]B$.
\end{lemma}

\begin{proof}
    Every closed substack of $BG_{\OO_K}$ is of the form $[Z/G]$ for a closed subscheme $Z$ of $\spec\OO_K$ by \cite[Exercise 3.3.9]{alpernotes}. If $L/K$ is a finite field extension, an $L$-point on $[Z/G]$ is given by a $G$-torsor $Y\to\spec L$ with a morphism of schemes $Y\to Z$. Since $\underline{G}$ is affine, $Y=\spec A$ for an \'etale $L$-algebra $A$ that is also an algebra over $\OO_K/I$, where $Z=\spec\OO_K/I$. If $I\neq(0)$, then the ring homomorphism $\ZZ\to \OO_K/I$ is not injective, but the ring homomorphism $\ZZ\to A$ is injective. Thus such a scheme $Y$ doesn't exist, and $[Z/G](\overline K)=\emptyset$.
\end{proof}

Let \(K\) be a number field, \(G\) a constant group scheme over \(K\), and \(x \in BG(K)\). If \(\mathcal{C} \to \spec \OO_K\) is a tuning stack corresponding to \(x\), 
define the expected deformation dimension to be 
    \[
   \edd(x) := \sum_{\substack{v \in M_K,\\ \text{ ramified in }  \mathcal{C}} } \log q_{v},
    \]
    where \(q_{v}\) denotes the size of the residue field at \(v\). This is a particular case of a more general definition, and seeks to imitate the anti-canonical height of a point. See \cite[\S 4.1]{ESZB}. 

Given a metrized vector bundle \(\mathcal{V}\) on \(BG_{\OO_K}\) {and $a'\in\RR$}, define
\begin{equation*}
\label{eq:Daprime}
    D_{a'}(x) = a'\height_{\mathcal{V}}(x) - \edd(x).
\end{equation*}

\begin{definition}
    \label{defn:a-invariant-ESZB}
    Define
    \[
    a(BG, \mathcal{V}) :=\inf \{a' \in \RR \mid D_{a'}(x) \text{ is generically bounded below on } BG(\ol{K})\}.
    \]
    If this infimum doesn't exist, then \(a(BG, \mathcal{V}):=-\infty\).
\end{definition}

\begin{conjecture}[Weak stacky Batyrev--Manin--Malle conjecture {\cite[Conjecture 4.14]{ESZB}}]
    Let \(\mathcal{X} =  {BG_{\OO_K}}\), 
    equipped with a vector bundle \(\mathcal{V}\). Consider for any open substack \(\mathcal{U} \subset \mathcal{X}\), the counting function 
\[
\countingESZB = \# \{ x \in \mathcal{X}(K) \mid \Height_{\mathcal{V}}(x) < B \}.
\]
Then there is a choice of \(\mathcal{U}\) such that for every \(\epsilon >0\), there is a constant \(c_{\epsilon}\) satisfying
\[
c_{\epsilon}^{-1}B^{a(\mathcal{X}, \mathcal{V})} \ll \countingESZB \ll c_{\epsilon} B^{a(\mathcal{X}, \mathcal{V}) + \epsilon}.
\]
\end{conjecture}

Of course, \cite[Conjecture 4.1.]{ESZB} is stated for more general \(\mathcal{X}\), namely for proper Artin stacks. In this article, we compute the invariant $a(B\mu_n, \mathcal{V})$. For a heuristic argument for \(BG\) for more general \(G\), we refer the reader to \cite{ESZB}.

\begin{proposition}
   Let \(K\) be a number field containing the \(n\)-th roots of unity, let \(\mathcal{X} = B\mu_n\) and let \(\mathcal{V}\) be the vector bundle corresponding to the regular representation as in Proposition \ref{proposition:ESZB-discriminant}. 
   Let \(r\) be the smallest prime dividing \(n\). Then 
   \[a(B\mu_n, \mathcal{V}) = \frac{2}{n-n/r}.
   \]
   \end{proposition}

Note that the factor of \(2\) appears here because the height in Proposition \ref{proposition:ESZB-discriminant} is the square root of the discriminant, rather than the discriminant.

\begin{proof}
Let $x\in B\mu_n(K)$ be a point corresponding to the $K$-algebra $K_a$ for some non-zero element $a\in \OO_{K}$. Combining Proposition \ref{proposition:ESZB-discriminant},  \cite[\S4.4]{ESZB} and Proposition \ref{prop:discriminant-Bmun}, we get
\begin{align*}
D_{a'}(x) &= a'\height_{\mathcal{V}}(x) - \edd(x) \\
& = \frac {a'}2 \log |\Delta(K_a/K)| - \sum_{v\mid \Delta(K_a/K)}\log q_v \\
& = \frac {a'}2 \left(\sum_{v\mid a, v\nmid n} (n-d_v)\log q_v + \sum_{v\mid n}e(a,v)\log q_v\right) - \sum_{v\mid \Delta(K_a/K)}\log q_v \\
& = \sum_{v\mid \Delta(K_a/K), v\nmid n}\left(\frac {a'}2 (n-d_v) -1 \right)\log q_v + \sum_{v\mid \Delta(K_a/K), v\mid n}(e(a,v)-1)\log q_v.
\end{align*}

Note that since \(0 \le e(a,v) \le C_{K,n}\) as in Proposition \ref{prop:discriminant-Bmun} and \(n\) is fixed, the second part of this sum is uniformly bounded by a constant for all \(x\). We will show that if $\frac {a'}2 (n-d_v) -1<0$ for some \(a' \in \RR\), some non-zero \(a \in \OO_K\) and some \(v \nmid n\) with  \(v \mid a\), 
then the first sum can be made arbitrarily negative on \(B\mu_n(K)\).

Observe that if $a=b^{d}$ with \(d\mid n\) and $b\in\ZZ$ square free, coprime to $n$, and such that every prime number that divides $b$ splits completely in $\OO_K$, then $n-d_v=n-d$ for every place of $K$ that divides $a$ but not $n$. 

If $\frac {a'}2 (n-d_v) -1<0$ for some $a'$, $a$, and $v$ as above, choose \(d\) such that \(d = d_v = \gcd(v(a), n)\). By the Chebotarev density theorem \cite{Chebotarev}, there are infinitely many prime numbers that split completely in $K$. So we can produce a sequence of integers $b$ as above where each element has more prime factors than the previous one, thus making the first part of \(D_{a'}(x)\) arbitrarily negative using \(x\) coming from \(b^d\).
Therefore, the function $D_{a'}$ is not bounded from below on $B\mu_n(K)$.
Thus $D_{a'}$ is generically bounded below on $B\mu_n(\overline K)$ if and only if 
$$
\frac {a'}2 \left(n-d_v\right) -1\geq 0
$$
for all $v\mid\Delta(K_a/K)$ such that $v \nmid n$
and for all $a\in \OO_{K}\smallsetminus\{0\}$, where $d_v=\gcd(v(a),n)$.
Since the infimum is the largest lower bound,
$$
a(B\mu_n,\mathcal V)=\max_{a\in \OO_{K,\neq0}}\max_{v\mid\Delta(K_a/K), v\nmid n} \frac 2 {n-d_v}.
$$
We observe that if $d_v\neq n$, then 
$$
n-d_v\geq n- n/r.
$$
While if $d_v=n$, then $n-d_v=0$, so that $v\nmid \Delta(K_a/K)$.
Moreover, if $a=b^{n/r}$ with $b\in\ZZ$ square free, coprime to $n$, and such that every prime number that divides $b$ splits completely in $\OO_K$, then $n-d_v=n-n/r$ for every place of $K$ that divides $a$. Thus, $a(B\mu_n,\mathcal V)=2/(n-n/r)$.
\end{proof}

\subsubsection{The raising function perspective}
\label{subsec:darda-yasuda-conjecture}
The current state-of-the-art conjecture for the \(b\)-invariant in the stacky analogue of Malle's conjecture comes from work of Darda and Yasuda in \cite{darda-yasuda22} and \cite{darda-yasuda-torsors1}. In this section, we discuss the computation of the \(a\) and \(b\) invariants from their work. 

Recall from \S \ref{subsec:darda-yasuda-heights}, that Darda and Yasuda define the height of a point on \(BG\) via raising data as in Definition \ref{defn:raising-function}, which captures inertia data of the corresponding \'etale \(G\)-algebra. The raising function \(c : \pi_0(\JetsX) \to \R_{\geq 0}\) plays an important role in their version of Malle's conjecture. More precisely, define,
\begin{align*}
     a(c) &:= \max\{c(\mathcal{Y})^{-1} \mid \mathcal{Y} \in {\pi_0}^*\left({\mathcal{J}_0\mathcal{X}}\right) \} \\
     b(c) &:= \#\{\mathcal{Y} \in {\pi_0}^*\JetsX \mid c(\mathcal{Y}) = a(c)^{-1} \}. 
\end{align*}
In \cite[Proposition 9.25]{darda-yasuda22}, the authors show that these invariants agree with
the definition of the \(a\) and \(b\) invariants for more general stacks via the effective cone defined in \cite[Definitions 9.3, 9.7]{darda-yasuda22}.

By \cite[Lemma 2.5.5]{darda-yasuda-torsors1}, if the local factors \(c_{v}\) are defined as the valuations
of the local discriminants, then the global function \(c\) is precisely the index as in Malle's conjecture. Since in the case where $\zeta_n \in K$,
the set of sectors \(\pi_0(\JetsX)\) can be identified with \(\ZZ/n\ZZ\), which in turn can be identified with the set of conjugacy classes in \(\ZZ/n\ZZ\), \(a(c)\) is precisely the minimal index \(1/(n-n/r)\).

\begin{remark}
    The proof of \cite[Proposition 9.24]{darda-yasuda22} gives an intuitive reason for why one might expect the definition of the \(a\)-invariant in \cite{ESZB} to be given via the condition of being generically bounded below. The latter condition imitates the condition of belonging to the effective cone as in Manin's conjecture. 
\end{remark}

\begin{conjecture}[{\cite[Conjecture 9.10]{darda-yasuda22}}, {\cite[Conjecture 2.6.5]{darda-yasuda-torsors1}}]
    Let \(\mathcal{X} = BG_{K}\). Then there is a thin set \(T\) as in \cite[\S 9.1]{darda-yasuda22} and a positive constant \(A\) such that
    \[
    \#\left\{x \in BG(K) \setminus T \mid \Height_c(x) < B \right\} \sim A \cdot B^{a(c)} \log(B)^{b(c) - 1}.
    \]
\end{conjecture}

\bibliographystyle{alpha}
\bibliography{bibliography}
\end{document}